\documentclass[11pt]{amsart}
\usepackage[utf8]{inputenc}
\usepackage{amsmath}
\usepackage{amsthm, amsfonts,srcltx,amsopn}
\usepackage{mathabx}
\usepackage{amssymb}
\usepackage[mathscr]{euscript}
\usepackage{enumerate}
\usepackage{braket}

\usepackage{paralist}

\usepackage{graphicx}
\usepackage[margin=1in]{geometry}

\usepackage{tikz}
  \usetikzlibrary{cd}
\usepackage{color}

\definecolor{Green}{RGB}{30, 150, 30}

\usepackage{multicol}
\usepackage{mathrsfs}
\usepackage{hyperref}
\usepackage{setspace}

\usetikzlibrary{arrows}
\tikzcdset{row sep/normal=1cm}
\tikzcdset{column sep/normal=1cm}
\newtheorem{thm}{Theorem}[section]
\newtheorem{claim}[thm]{Claim}
\newtheorem{prop}[thm]{Proposition}
\newtheorem{lem}[thm]{Lemma}
\newtheorem{cor}[thm]{Corollary}

\theoremstyle{definition}
\newtheorem{defn}[thm]{Definition}

\newtheorem{question}[thm]{Question}
\newtheorem{rem}[thm]{Remark}
\newtheorem{notation}[thm]{Notation}
\newtheorem{StandingAssumption}[thm]{Standing Assumption}

\newcommand*\Z{\mathbb{Z}}

\newcommand*\supp{\operatorname{supp}}

\newcommand*\gate{\mathfrak{g}}

\newcommand*\diam{\operatorname{diam}}

\newcommand*\nest{\sqsubseteq}
\newcommand*\propnest{\sqsubsetneq}

\newcommand*\mf[1]{\mathfrak {#1}}
\newcommand*\mc[1]{\mathcal {#1}}
\newcommand*\trans{\pitchfork}
\newcommand*\BS{\operatorname{Big}}

\newcommand{\ol}{\overline}
\def\P{\mathbf{P}}
\def\F{\mathbf{F}}
\def\E{\mathbf{E}}

\newcommand*\gp[2]{({#1} \mid {#2})_{x_0}}
\newcommand*\gprel[3]{({#1} \mid {#2})_{#3}}
\newcommand*\gph[2]{\langle {#1} \mid {#2} \rangle_{x_0}}

\newcounter{jcomments}

\newcounter{jrcomments}

\newcounter{ccomments}

\long\def\Restate#1#2#3#4{
\medskip\par\noindent
{\bf #1 \ref{#2} #3} {\it #4}\par\medskip }

\newcommand{\s}{\mf{S}}

\title[Structure invariant properties of the hierarchically
hyperbolic boundary]{Structure invariant properties of the\\ hierarchically
hyperbolic boundary}

\author{Carolyn Abbott}
\address{Brandeis University, Waltham, MA, USA}
\email{carolynabbott@brandeis.edu}

\author{Jason Behrstock}
\address{Lehman College and The Graduate Center, CUNY, New York, New York, USA}
\curraddr{Barnard College, Columbia University, New York, New York, USA}
\email{jason@math.columbia.edu}

\author{Jacob Russell}
	\address{Department of Mathematics, Rice University, Houston, TX}
\email{jacob.russell@rice.edu}

\begin{document}

\begin{abstract}
	We prove  several topological and dynamical properties of the boundary of a hierarchically hyperbolic group are independent of the specific hierarchically hyperbolic structure. This is accomplished by proving that the boundary is invariant under a ``maximization'' procedure introduced by the first two authors and Durham.
\end{abstract}

\maketitle

\section{Introduction}

A geodesic metric space $\mc X$ has a \emph{hierarchically hyperbolic
structure}  if there exist an index set $\mf{S}$ parameterizing a
collection of hyperbolic spaces and projections from $\mc X$
to each of these hyperbolic spaces satisfying some conditions that
encode the presence/absence of certain quasi-isometrically embedded
products in $\mc X$; see Definition \ref{defn:HHS}.  A
\emph{hierarchically hyperbolic group} (HHG) is a finitely generated group
where the word metric on the group has a hierarchically hyperbolic
structure that is compatible with the group action.

The notion of hierarchical hyperbolicity was introduced by
Behrstock, Hagen, and Sisto \cite{BHS_HHSI,BHS_HHSII} and includes a
number of important examples in geometric group theory including
mapping class groups, most three-manifold groups \cite{BHS_HHSII},
right-angled Coxeter groups, large classes of Artin groups
\cite{BHS_HHSI,HagenMartinSisto:XLartinHHS}, hyperbolic groups, and
various combinations of these examples
\cite{BHS_HHSII,BR_Combination,BR_graph_products}.

A hierarchically hyperbolic group is typically studied by fixing a
particular hierarchically hyperbolic structure and deducing results
about the group using the geometric and combinatorial properties
associated to that structure.  On the other hand, once a group admits
one hierarchically hyperbolic structure it automatically admits many.
For example, a hyperbolic group has a hierarchically hyperbolic
structure where the index set consists of a single element and the
associated hyperbolic space is the group itself.  However, one can also take a
more complicated index set consisting of a collection of quasi-convex
subgroups of the hyperbolic group together with an electrification of
the original group collapsing those subgroups and their cosets.  (Note
that a hyperbolic group is not hyperbolic relative to such a
collection if they are not almost malnormal, but one does get
hierarchical hyperbolicity without such an assumption.)  This point of
view provides an upside to having multiple structures by yielding new
techniques for studying hyperbolic groups and their boundaries; see 
e.g., 
\cite{Spriano_hyperbolic_I}.

In a few cases, the possible structures for an HHG are understood; for instance a
hierarchically hyperbolic group is virtually abelian if and only if
the associated hyperbolic spaces are either bounded or quasi-lines
\cite{PS_Unbounded_domains}.  However, at this point,  it
remains out of reach to understand all the possible structures on a
given group in general.  A natural question in this direction is whether or not a
hierarchically hyperbolic group possesses a most ``natural'' or
``simplest'' structure.  In a hierarchically hyperbolic structure, there is a partial order, called \emph{nesting}, on the set of hyperbolic spaces and a unique nest-maximal element.  Understanding the geometry of the nest-maximal hyperbolic space  is one way to make precise the notion of a simplest structure. Some progress on this has been made in \cite{ABD}, where the authors
gave a construction, which we call \emph{maximization}, that modifies
any given structure to produce one where the nest-maximal hyperbolic
space is canonical.  The canonicality can be seen in a number of ways,
including being unique up to quasi-isometry, as well as encoding all
the Morse elements of the group.  This paper begins with the work of
\cite{ABD} as a starting point in order to study the effect of
maximization on the boundary of a hierarchically hyperbolic group.

Durham, Hagen, and Sisto introduced a 
boundary that provides a compactification for a
hierarchically hyperbolic group and coincides with the Gromov 
boundary when the group is hyperbolic \cite{HHS_Boundary}. Their construction depends {\it a priori} on the 
choice of hierarchically hyperbolic structure $\frak S$ for the group $G$, a pair which we denote $(G,\frak S)$; accordingly, we denote this boundary $\partial (G,\frak S)$.  Question 1 in Durham, Hagen, and Sisto's paper is: 
given two different structures on a hierarchically hyperbolic 
group, does the identity map from  the group to itself extend to 
a homeomorphism between the boundaries of the two different structures? 

In this paper we resolve  
Durham--Hagen--Sisto's question for 
any structure and its maximized version. 

\Restate{Theorem}{thm:BdryABDInvariant}{} {If $(G,\mf{S})$ is an HHG
	and $\mf{T}$ is the structure obtained by maximizing $\mf{S}$, then
	the identity map on $G$ extends continuously to a $G$--equivariant map $ \partial(G,\mf{S}) \to
	\partial(G,\mf{T})$ that is both a simplicial isomorphism and a
	homeomorphism.  }

As part of our proof of Theorem \ref{thm:BdryABDInvariant}, we also prove that two other important notions in hierarchical hyperbolicity---hierarchical quasiconvexity and hierarchy paths---are also invariant under the maximization procedure; see Section \ref{sec:hqcinvar} for the precise statements.

Theorem~\ref{thm:BdryABDInvariant} allows one to convert questions 
about the HHG boundary to questions about a maximized structure. In 
particular, this 
allows us to obtain a number of results about HHG boundaries which  
are independent of the choice of HHG structure used to build the 
boundary.

One consequence is that some topological properties of the boundary of
the maximized hyperbolic space can be shown to hold in every HHG 
boundary, for instance:

\Restate{Corollary}{cor:boundary_connected}{} {Let $G$ be an HHG. 
	If the hyperbolic space  associated to the 
	nest-maximal element  in some (and hence any) maximized 
	hierarchically hyperbolic structure is one-ended, then for any HHG
	structure $\mf{S}$ for $G$, the HHS boundary $\partial(G,\mf{S})$ is
	connected.}

\noindent The converse of Corollary~\ref{cor:boundary_connected} is an interesting open question.

Theorem~\ref{thm:BdryABDInvariant} and the fact that the maximized
hyperbolic space encodes the Morse elements of the group implies that the Morse
elements are precisely the set of elements that act with north-south
dynamics in \textit{any} HHG boundary.

\Restate{Corollary}{cor:boundary_NSdynamics}{} {Let $(G,\mf{S})$ be a
	hierarchically hyperbolic group that is not virtually cyclic.  An
	element $g \in G$ acts with north-south dynamics on
	$\partial(G,\mf{S})$ if and only if $g$ is a Morse element of $G$.  In
	particular, the set of elements of $G$ that act with north-south
	dynamics does not depend on the HHG structure $\mf{S}$.}

We can also show that the set of attracting fixed points of the Morse elements is dense in the boundary, regardless of the choice of HHG structure.

\Restate{Corollary}{cor:Morse_boundary_is_dense}{} {Let $(G,\mf{S})$
	be an HHG that is not virtually cyclic.  
	Either $G$ is quasi-isometric to a product of two unbounded spaces 
	or  the set of attracting fixed points of Morse elements in  
	$\partial(G,\mf{S})$ is dense in $\partial (G,\mf{S})$. In the latter case, the 
	Morse boundary is a dense subset of the HHS boundary.}

Note that for the above corollary, as well as the following one, the hypothesis that $G$ is 
quasi-isometric to a product of two unbounded spaces could be 
replaced by the equivalent statement that  
$G$ does contain a Morse element. This equivalence is obtained by 
first applying the rank rigidity theorem \cite[Theorem~9.13]{HHS_Boundary} to know 
that a group is either a product or contains a rank-one element, then 
applying maximization to ensure that a rank-one element 
is irreducible axial, and  finally appealing to \cite[Theorem~6.15]{HHS_Boundary} (or 
alternatively \cite[Theorem~4.4]{ABD}), which implies that irreducible 
axials are Morse.

Finally, we use the density of Morse elements to show the limit set of
a normal subgroup is the entire HHS boundary.  Examples of such normal
subgroups include the kernel of the Birman exact sequence \cite{birmanbraids},
Bestvina--Brady subgroups of RAAGs \cite{Bestvin_Brady_Morse_Theory},
the normal closure of sufficiently high powers of Dehn twists, as
studied in \cite{Dahmani:plate_spinning}, and the infinitely generated
RAAG subgroups of mapping class groups considered in
\cite{ClayMangahasMargalit}.

\Restate{Corollary}{cor:normal_subgroups}{} {Let $(G,\mf{S})$ be an
	HHG that is not quasi-isometric to a product of two unbounded spaces 
	and 
	is not virtually cyclic.  If $N$ is an infinite normal subgroup of $G$, 
	then the limit set of $N$ in
	$\partial (G,\mf{S})$ is all of $\partial (G,\mf{S})$.}

For hyperbolic groups, if the normal subgroup $N$ is also hyperbolic,
then a remarkable theorem of Mj says there is a continuous surjection
of the Gromov boundary of $N$ onto the Gromov boundary of the ambient
group $G$ induced by the (highly distorted) inclusion of $N$ into $G$
\cite{Mj_Cannon-Thurston}.  These maps are often called
\emph{Cannon--Thurston maps} in honor of the fact that they were first
discovered by Cannon and Thurston in the case where $G$ is the fundamental
group of a fibered hyperbolic $3$--manifold
\cite{Cannon-Thurston_Original}.  Corollary
\ref{cor:normal_subgroups} therefore inspires the following question.

\begin{question}\label{ques:cannon-thurston-map}
	For which HHGs do Cannon--Thurston maps exist?  That is, if
	$(G,\mf{S})$ is an HHG and $N$ is a normal subgroup of $G$ that
	has an HHG structure $\mf{T}$, when does the inclusion $N \to G$
	induce a continuous surjection $\partial(N,\mf{T}) \to \partial
	(G,\mf{S})$?
\end{question}

Natural test cases of Question \ref{ques:cannon-thurston-map} are the
kernel of the Birman exact sequence and cases when a Bestvina--Brady
group is itself a RAAG. 
The answer is ``no'' when $G$ is the direct
product of two hyperbolic groups, but no other obstructions are
currently known.  

A more general formulation of
Question~\ref{ques:cannon-thurston-map} is to remove the normal
hypothesis and ask for which hierarchically hyperbolic subgroups there
is a continuous extension from the HHS boundary of the subgroup to its
limit set.  This version holds for quasiconvex subgroups in hyperbolic
groups.  Moreover, the naive obstruction noted above provided by the
product of two hyperbolic groups is not an obstruction to this version
of the question.  On the other hand, Mousley showed that for a family
of RAAG subgroups of mapping class groups, there are obstructions to
extending, while for a family of free groups she characterizes exactly
when they do extend \cite{Mousley:noboundarymapsHHS}.


\subsection*{Organization of the paper}
In Section \ref{sec:hyp_background}, we set  notation and collect preliminary facts on hyperbolic spaces and their boundaries. In Section \ref{sec:HHS_background}, we define hierarchically hyperbolic spaces and their boundaries. We also describe several tools from the theory of hierarchical hyperbolicity that we will use. We begin Section \ref{sec:maximization} by describing the maximization procedure in detail (Section \ref{sec:maximizationsteps}), and then devote the remainder of the section to the proof of Theorem \ref{thm:BdryABDInvariant}. Section \ref{sec:applications} contains the applications of Theorem \ref{thm:BdryABDInvariant} including Corollaries \ref{cor:boundary_connected}, \ref{cor:boundary_NSdynamics}, \ref{cor:Morse_boundary_is_dense}, and \ref{cor:normal_subgroups}. Section \ref{sec:maximization} contains the bulk of the technical work of the paper. Section \ref{sec:applications} is an essentially self-contained collection of applications where the only reference to the rest of the paper is the statement of  Theorem \ref{thm:BdryABDInvariant}.

\subsection*{Acknowledgments}
We thank Mark Hagen for helpful discussions, especially 
concerning the topology of the hierarchically hyperbolic
boundary. We thank the anonymous referee for comments that improved the exposition of the paper.
 Abbott was supported by NSF grants DMS-1803368 and 
DMS-2106906. Behrstock was supported by the Simons Foundation as a 
Simons Fellow. Behrstock thanks the Barnard/Columbia Mathematics 
department for their hospitality. Russell was supported by NSF grant DMS-2103191.

\section{Preliminaries on hyperbolic spaces} \label{sec:hyp_background}

In this section, we establish notation and recall some basic notions 
about hyperbolic spaces. We also deduce a few general results 
about hyperbolic spaces, which  we will use later in the paper. 

\subsection{Coarse Geometry}
We begin by gathering several facts about metric spaces and coarse 
geometry. We refer the reader to \cite{BridsonHaefliger:book} for further details.

Let $(X,d_X)$ be a metric space. If $Y\subseteq X$ is a subspace, then for any constant $C\geq 0$,  we denote the closed $C$--neighborhood of $Y$ in $X$ by
\[\mathcal N_C(Y)=\{x\in X\mid d_X(x,Y)\leq C\}.\]      We say two  subsets $Y,Z \subseteq X$ are \emph{$C$--coarsely equal}, for some $C\geq 0$, if $Y \subseteq \mc{N}_C(Z)$ and $Z  \subseteq \mc{N}_C(Y)$. When $Y$ and $Z$ are $C$--coarsely equal, we write $Y\asymp_C Z$.  

A map of metric spaces $f\colon (X,d_X)\to (Y,d_Y)$ is a \emph{$(\lambda,c)$--quasi-isometric embedding} if for all $x,y\in X$
\[
\frac1\lambda d_X(x,y)-c\leq d_Y(f(x),f(y))\leq \lambda d_X(x,y)+c.
\]
A \emph{$(\lambda,c)$--quasigeodesic} is a
$(\lambda,c)$--quasi-isometric embedding of a closed interval
$I\subseteq \mathbb R$ into $X$, and a \emph{geodesic} is an isometric
embedding of $I$ into $X$.  We let $[x,y]$ denote a geodesic in $X$
from $x$ to $y$.  In the case of quasigeodesics, we allow $f$ to be a
\emph{coarse map}, that is, a map which sends points in $I$ to
uniformly bounded diameter sets in $X$.  Accordingly, we can assume
that the domain of a quasigeodesic is an interval in $\mathbb{Z}$
instead of $\mathbb R$ when convenient.  A (coarse) map $f\colon
[0,T]\to X$ is an \emph{unparametrized $(\lambda,c)$--quasigeodesic}
if there exists a non-decreasing function $g\colon
[0,T']\to[0,T]$ such that the following hold:
	\begin{itemize} 
	\item $g(0)=0$, 
	\item $g(T')=T$,  
	\item $f\circ g\colon [0,T']\to X$ is a $(\lambda,c)$--quasigeodesic.
	\item for each $j\in[0,T']\cap \mathbb N$, we have the diameter of
	$f(g(j))\cup f(g(j+1))$ is at most $c$.  
	\end{itemize}

For  $\delta\geq 0$, a geodesic metric space $X$ is \emph{$\delta$--hyperbolic} if we have $[x,y] \subseteq \mc{N}_\delta( [x,z] \cup [z,y])$. If the particular choice of $\delta$ is not important, we simply say that $X$ is \emph{hyperbolic}.

Quasigeodesics in a hyperbolic metric space satisfy the following
Morse property, which roughly states that quasigeodesics with the same
endpoints remain in a uniform neighborhood of each other.  This is
also known as \emph{quasigeodesic stability}.

\begin{lem}[Morse Lemma]\label{lem:Morse}
Let $X$ be a $\delta$--hyperbolic metric space, and fix $\lambda\geq 1$ and $c\geq0$.  There exists a constant $\sigma$ depending only on $\delta,\lambda$, and $c$ such that if $\gamma_1$ and $\gamma_2$ are $(\lambda,c)$--quasigeodesics in $X$ with the same endpoints, then $\gamma_1 \asymp_\sigma \gamma_2$.
\end{lem}
We say $\sigma$ is the \emph{Morse constant} associated to $( \lambda,c)$--quasigeodesics in a $\delta$--hyperbolic~space.

\subsection{The Gromov product and the Gromov boundary}
Let $X$ be a $\delta$--hyperbolic metric space. For any $x,y,z \in X$, the \emph{Gromov product of $x$ and $y$ with respect to $z$} is \[ \gprel{x}{y}{z} = \frac{1}{2} \left(d_X(x,z) + d_X(y,z) - d_X(x,y) \right).\]

The Gromov product $\gprel{x}{y}{z}$ is uniformly close  to the distance from $z$ to a geodesic connected $x$ and $y$:
\begin{lem}[{\cite[p.~410]{BridsonHaefliger:book}}] \label{lem:Gromov_product_is_distance_to_geodesic}
	For any $\delta$--hyperbolic space $X$ and $x,y,z \in X$, we have \[ |\gprel{x}{y}{z}- d_X(z, [x,y]) | \leq \delta.\]
\end{lem}

Given a fixed basepoint $x_0$ of $X$, a sequence of points $(x_n)$ in $X$  \emph{converges to infinity} if $$\gp{x_n}{x_k}  \to  \infty$$ as
 $n,k \to \infty$. Two sequences $(x_n)$ and $(y_n)$ are \emph{asymptotic} if $\gp{x_n}{y_n} \to \infty$ as $n \to \infty$. Note, this is equivalent to requiring that $\gp{x_n}{y_k} \to \infty$ as $n,k \to \infty$.  The \emph{Gromov boundary} $\partial X$ of  $X$ is the set of sequences in $X$ that converge to infinity modulo the equivalence relation of being asymptotic. 

The Gromov product extends to $x,y \in X \cup \partial X$ and $z \in X$ by taking the supremum of $$\liminf\limits_{n,k}\gprel{x_n}{y_k}{z}$$ over all sequences $(x_n)$ and $(y_k)$ that are either asymptotic to $x$ or $y$ when they are boundary points or converge to $x$ or $y$ when they are points in $X$. We can then topologize $X \cup \partial X$ by declaring a sequence $(x_n)$ in $X \cup \partial X$ to converge to $x \in \partial X$ if and only if $$\lim_{n \to \infty} \gp{x_n}{x} = \infty.$$	

\begin{defn}\label{def:nbhdbasis}
For each $p \in \partial X$, the sets 
\[
M(r;p) = \left\{ x \in X \cup \partial X : \gp{p}{x} > r\right\} \] where
$r >0$ form a neighborhood basis for $p$ in $X \cup \partial X$.  Note that
if $r \leq r'$, then $M(r';p) \subseteq M(r;p)$.
 \end{defn}

Despite the presence of the basepoint in the above definitions, convergence to infinity, being asymptotic, the Gromov boundary, and the topology of $X \cup \partial X$ are all independent of the choice of basepoint.

If $\gamma \colon [0,\infty) \to X$ is a quasigeodesic ray, then
there exists a unique $p \in \partial X$ so that $\gamma(t_n) \to p$
for every increasing sequence $(t_n)$ that approaches infinity.  In
this case, we say that $\gamma$ \emph{represents} $p \in \partial X$.
Every point in $\partial X$ can be represented by a
$(1,20\delta)$--quasigeodesic ray based at any point in $X$; see 
e.g., \cite{KapovichBenakli:boundaries}.

The next two lemmas allow geodesics to assist in calculating  Gromov products.  The first is straight-forward, and its proof is left to the reader.

\begin{lem}\label{lem:Gromov_product_along_qg}
	For every $\lambda\geq 1$, $c \geq 0$, and $\delta \geq 0$, there exist  $B \geq 0$ and $t_0 \geq 0$ so that the following holds.
	
	Let $X$ be a $\delta$--hyperbolic space with basepoint $x_0$, and let $p\in \partial X$ and $x \in X$. Let $\gamma \colon [0,\infty) \to X$ be a $(\lambda,c)$--quasigeodesic starting at $x_0$ and representing $p$. For all $t \geq t_0$, we have \[|\gp{\gamma(t)}{x}  - \gp{p}{x}| \leq B. \] 
\end{lem}

\begin{lem}\label{lem:Gromov_product_go_far_enough}
	Let $X$ be a $\delta$--hyperbolic space and $x,y,z \in X$. If $q$ is a point on $[z,x]$ so that $q \in \mc{N}_C\left([x,y]\right)$, then  \[ |\gprel{x}{y}{z} - \gprel{q}{y}{z}| \leq C.\]
\end{lem}

\begin{proof}
	Since $q \in [z,x]$, we have \[d_X(q,z) +d_X(y,z) - d_X(q,y) = d_X(x,z) - d_X(q,x) + d_X(y,z) -d_X(q,y). \]
	Let $p$ be point on $[x,y]$ with $d_X(p,q) \leq C$. Thus, \[d_X(q,y) \leq d_X(p,y) + C  \text{ and } d_X(q,x) \leq d_X(p,x) + C. \]
	Hence,  
	\begin{align*}
		d_X(x,z) - d_X(q,x) + d_X(y,z) - d_X(q,y) &\geq d_X(x,z) - d_X(p,x)  + d_X(y,z) - d_X(p,y) - 2C \\
		&=  d_X(x,z)   + d_X(y,z) - d_X(x,y) - 2C. 
	\end{align*}
	For the other inequality, we apply \[d_X(q,y) \geq d_X(p,y) - C  \text{ and } d_X(q,x) \geq d_X(p,x) - C \]
	to conclude 
	\begin{align*}
		d_X(x,z) - d_X(q,x) + d_X(y,z) - d_X(q,y) &\leq d_X(x,z) - d_X(p,x)  + d_X(y,z) - d_X(p,y) + 2C \\
		&=  d_X(x,z)   + d_X(y,z) - d_X(x,y) + 2C. 
	\end{align*}
	Dividing everything by $2$ produces $|\gprel{x}{y}{z} - \gprel{p}{y}{z}| \leq C$.	
\end{proof}

Our last lemma shows that  points of $X$ that are close to $M(r;p)$ are in $M(r';p)$ for some $r'$ slightly smaller than $r$.

\begin{lem}\label{lem:enlarging_boundary_neighborhoods}
	Let $X$ be a $\delta$--hyperbolic space. There exists a constant $B \geq 0$, depending only on $\delta$, so that for all $p\in\partial X$, if  $C \geq 3\delta +2B$ and  $r \geq 2C +1$, then $$\mc{N}_C(M(r;p) \cap X) \subseteq M(r-2C;p).$$  
\end{lem}

\begin{proof}
	Let $B$ be the constant from Lemma \ref{lem:Gromov_product_along_qg} for $\lambda =1$ and $c = 20\delta$. Let $C \geq 3\delta +2B$, $r \geq 2C+1$, and $p\in \partial X$.
	
	 Select a a $(1,20\delta)$--quasigeodesic  $\gamma$ starting at the basepoint $x_0$ and representing $p$. There is a point $y \in \gamma$ that is sufficiently far along $\gamma$ so that $y \in M(r;p)$ and Lemma \ref{lem:Gromov_product_along_qg} applies to $y$. Let $x \in \mc{N}_C(M(r;p) \cap X)$ and $x' \in M(r;p) \cap X$ with $d(x,x') \leq C$. Using Lemmas \ref{lem:Gromov_product_along_qg} and \ref{lem:Gromov_product_is_distance_to_geodesic} we have 
	\begin{equation}\label{eq:gp_calc}
		r \leq \gp{x'}{p} \leq \gp{x'}{y} +B \leq d(x_0, [x',y]) +\delta +B. 
	\end{equation} Since $d(x,x') \leq C$, $\delta$--hyperbolicity of $X$ implies $d(x_0,[x',y]) \leq d(x_0,[x,y])+C +\delta$. Combining this with  \eqref{eq:gp_calc} yields $$  r \leq \gp{x'}{p}  \leq d(x_0,[x,y]) + C +2\delta +B.$$  When $C \geq 3\delta +2B$, this shows $x \in M(r-2C;p)$  as  $$d(x_0,[x,y]) \leq \gp{x}{y} +B \leq \gp{x}{p} +\delta + B$$ by Lemmas \ref{lem:Gromov_product_along_qg} and \ref{lem:Gromov_product_is_distance_to_geodesic}.	
\end{proof}

\subsection{Quasiconvex subsets}

A subset $Y$ of a $\delta$--hyperbolic space $X$ is \emph{$\mu$--quasiconvex} if every geodesic in $X$ between points in $Y$ is contained in the closed $\mu$--neighborhood of $Y$. We recall a few basic facts about quasiconvex subsets of hyperbolic spaces and verify a simple lemma. We direct  the reader to~\cite[\S 11.7]{DrutuKapovichBook} for full details. 

When $Y$ is a $\mu$--quasiconvex subset of a $\delta$--hyperbolic
space $X$ and $x\in X$, the set of points $\{y \in Y : d_X(x,y) \leq
d_X(x,Y) +1\}$ is uniformly bounded in terms of $\mu$ and $\delta$.
Hence, there is a well defined coarse map $\mf{p}_Y \colon X \to Y$ so
that $$\mf{p}_{Y}(x) = \{y \in Y : d_X(x,y) \leq d_X(x,Y) +1\}.$$ We call the map
$\mf{p}_{Y}$ the \emph{closest point projection onto $Y$}.

For a quasiconvex subset $Y \subseteq X$, we let $\partial Y$ denote the set of points in $\partial X$ that are represented by sequences of points in $\partial Y$. The following lemma shows that quasigeodesics in $X$ that represent points in $\partial Y$ can be modified to be eventually contained in $Y$.
\begin{lem}\label{lem:qgeonearY}
Let $Y$ be a $\mu$--quasiconvex subset of a $\delta$--hyperbolic space
$X$.  Let $\gamma\colon [0,\infty)\to X$ be a
$(1,20\delta)$--quasigeodesic ray from a point $x\in X$ to a point
$p\in\partial Y$.  There then exists a constant $A\geq 1$, depending
only on $\delta$ and $\mu$, such that the following holds.  There is a
$(1,20\delta+2A)$--quasigeodesic ray $\gamma'\colon (0,\infty)\cap
\Z\to X$ from $x$ to $p$ and a constant $T\in [0,\infty)\cap \Z$ such
that $\gamma'(t)\in Y$ for all $t\geq T$ and $\gamma'$ is uniformly 
close to $\gamma$. 
\end{lem}

\begin{proof}
Since $p\in \partial Y$ and $Y$ is $\mu$--quasiconvex, there is a constant $A$, depending only on $\delta$ and $\mu$, and  some $t_0\in (0,\infty)$ such that $\gamma|_{[t_0,\infty)}\subseteq \mc N_{A}(Y)$.   Therefore, we may assume without loss of generality that $t_0\in \Z_{\geq 0}$. For each $i \in \mathbb{Z}_{>0}$, define $t_i=t_0+i$, and choose a point $y_i\in Y$ such that $d_X(\gamma(t_i),y_i)\leq A$.

Define $\gamma'\colon [0,\infty)\cap \Z\to X$ by 
\[
\gamma'(t)=\begin{cases}
\gamma(t) & \textrm{ if } t\in [0,t_0]\cap \Z \\
y_{t-t_0} & \textrm{ if } t\in (t_0,\infty)\cap \Z.
\end{cases}
\]

We will show that $\gamma'$ is a $(1,20\delta + 2A)$--quasigeodesic.  Let $t,s\in [0,\infty)\cap \Z$.  If $t,s\in [0,t_0]$, then the result is clear.  Suppose $t,s\in (t_0,\infty)\cap \Z$.  Then $d_X(\gamma'(t),\gamma'(s))=d_X(y_{t-t_0},y_{s-t_0})$, and  
\begin{align*}
d_X(\gamma(t-t_0),\gamma(s-t_0))-2A&\leq d_X(y_{t-t_0},y_{s-t_0})\leq d_X(\gamma(t-t_0),\gamma(s-t_0))+2A  \\
|(t-t_0)-(s-t_0)|-20\delta-2A & \leq d_X(y_{t-t_0},y_{s-t_0}) \leq |(t-t_0)-(s-t_0)|+20\delta+2A.
\end{align*}
Since $|(t-t_0)-(s-t_0)|=|t-s|$, we conclude that $\gamma'|_{(t_0,\infty)}$ is a $(1,20\delta+2A)$--quasigeodesic.

Finally, suppose $t \in[0,t_0]$ and $s\in (t_0,\infty)$.  Then $\gamma'(s)=y_{s-t_0}$ and $t_{s-t_0}=t_0+(s-t_0)=s$, and so $d_X(y_{s-t_0},\gamma(s))=d_X(y_{s-t_0},\gamma(t_{s-t_0}))\leq A$.  Thus we have \[d_X(\gamma(t),\gamma(s)-A\leq d_X(\gamma'(t),\gamma'(s))=d_X(\gamma(t),y_{s-t_0})\leq d_X(\gamma(t),\gamma(s))+A.
\]
Therefore, $\gamma'$ is a $(1,20\delta+2A)$--quasigeodesic which, by
construction, is from $x$ to $p$ and is uniformly close to $\gamma$, completing the proof.
\end{proof}

\section{Preliminaries on hierarchically hyperbolic spaces}\label{sec:HHS_background}
In this section, we recall some of the tools we will use to work with
hierarchically hyperbolic spaces and groups and define the HHS
boundary.  We begin with the definition of an HHS.

\begin{defn}[HHS]\label{defn:HHS}
	Let $E>0$ and $\mc{X}$ be an $(E,E)$--quasigeodesic space.  A
	\emph{hierarchically hyperbolic space structure with constant $E$}
	for $\mc{X}$ is an index set $\mathfrak S$ and a set $\{ \mc{C}W :
	W\in\mathfrak S\}$ of $E$--hyperbolic spaces $(\mc{C}W,d_W)$ such
	that the following axioms are satisfied.  \begin{enumerate}
		
	\item\textbf{(Projections.)}\label{axiom:projections} For each $W \in \mf{S}$, there exists a \emph{projection} $\pi_W \colon \mc{X} \rightarrow 2^{\mc{C}W}$ such that for all $x \in \mc{X}$, $\pi_W(x) \neq \emptyset$ and $\diam(\pi_W(x)) \leq E$. Moreover, each $\pi_W$ is $(E, E)$--coarsely
		Lipschitz and $\mc{C}W \subseteq \mc{N}_E(\pi_{W}(\mc{X}))$ for all $W \in \mf{S}$.

		\item \textbf{(Nesting.)} \label{axiom:nesting} If $\mathfrak S \neq \emptyset$, then $\mf{S}$ is equipped with a  partial order $\nest$ and contains a unique $\nest$--maximal element. When $V\nest W$, we say $V$ is \emph{nested} in $W$.  For each
		$W\in\mathfrak S$, we denote by $\mathfrak S_W$ the set of all $V\in\mathfrak S$ with $V\nest W$.  Moreover, for all $V,W\in\mathfrak S$ with $V\propnest W$ there is a specified non-empty subset $\rho^V_W\subseteq \mc{C}W$ with $\diam(\rho^V_W)\leq E$, and a map $\rho_V^W \colon \mc{C}W - \mc{N}_E(\rho_W^V) \to 2^{\mc{C}V}$.

		\item \textbf{(Orthogonality.)} 
		\label{axiom:orthogonal} $\mathfrak S$ has a symmetric relation called \emph{orthogonality}. If $V$ and $W$ are orthogonal, we write $V\perp
		W$ and require that $V$ and $W$ are not $\nest$--comparable. Further, whenever $V\nest W$ and $W\perp
		U$, we require that $V\perp U$. We denote by $\mf{S}_W^\perp$ the set of all $V\in \mf{S}$ with $V\perp W$.

		\item \textbf{(Transversality.)}
		\label{axiom:transversality} If $V,W\in\mathfrak S$ are not
		orthogonal and neither is nested in the other, then we say $V,W$ are
		\emph{transverse}, denoted $V\trans W$.  Moreover, for all $V,W \in \mathfrak{S}$ with $V\trans W$ there are non-empty
		sets $\rho^V_W\subseteq \mc{C}W$ and
		$\rho^W_V\subseteq \mc{C} V$ each of diameter at most $E$.

		\item \textbf{(Finite complexity.)} \label{axiom:finite_complexity} Any set of pairwise $\nest$--comparable elements has cardinality at most $E$.
		
		\item \textbf{(Containers.)} \label{axiom:containers}  For each $W \in \mf{S}$ and $U \in \mf{S}_W$ with $ \mf{S}_W\cap \mf{S}_U^\perp \neq \emptyset$, there exists $Q \in\mf{S}_W$ such that $V \nest Q$ whenever $V \in\mf{S}_W \cap \mf{S}_U^\perp$.  We call $Q$ the \emph{container of $U$ in $W$}.
		
		\item\textbf{(Uniqueness.)} There exists a function 
		$\theta \colon [0,\infty) \to [0,\infty)$ so that for all $r \geq 0$, if $x,y\in\mc X$ and
		$d_\mc{X}(x,y)\geq\theta(r)$, then there exists $W\in\mathfrak S$ such
		that $d_W(\pi_W(x),\pi_W(y))\geq r$. \label{axiom:uniqueness}
		
		\item \textbf{(Bounded geodesic image.)} \label{axiom:bounded_geodesic_image} For all  $V,W\in\mathfrak S$ with $V\propnest W$, if a $\mc{C} W$ geodesic $\gamma$ does not intersect $\mc{N}_E(\rho_W^V)$, then $\diam_{\mc{C}V}(\rho_V^W(\gamma)) \leq E$.

		\item \textbf{(Large links.)} \label{axiom:large_link_lemma} 
		For all $W\in\mathfrak S$ and  $x,y\in\mc X$, there exists $\{V_1,\dots,V_m\}\subseteq\mathfrak S_W -\{W\}$ such that $m$ is at most $E d_{W}(\pi_W(x),\pi_W(y))+E$, and for all $U\in\mathfrak
		S_W - \{W\}$, either $U\in\mathfrak S_{V_i}$ for some $i$, or $d_{U}(\pi_U(x),\pi_U(y)) \leq E$.  
		
		\item \textbf{(Consistency.)}
		\label{axiom:consistency} For all $x \in\mc X$ and $V,W,U \in\mf{S}$:
		\begin{itemize}
			\item  if $V\trans W$, then $\min\left\{d_{
				W}(\pi_W(x),\rho^V_W),d_{
				V}(\pi_V(x),\rho^W_V)\right\}\leq E$,
			\item if $V \propnest W$, then $ \min\left\{ d_W(\pi_W(x), \rho_V^W), \diam\bigl(\pi_V(x) \cup \rho_W^V(\pi_V(x))\bigr) \right\} \leq E$,
			\item if $U\nest V$ and either $V\propnest W$ or $V\trans W$ and $W\not\perp U$, then $d_W(\rho^U_W,\rho^V_W)\leq E$.
		\end{itemize}

		\item \textbf{(Partial realization.)} \label{axiom:partial_realisation}  If $\{V_i\}$ is a finite collection of pairwise orthogonal elements of $\mathfrak S$ and $p_i\in  \mc{C}V_i$ for each $i$, then there exists $x\in \mc X$ so that:
		\begin{itemize}
			\item $d_{V_i}(\pi_{V_{i}}(x),p_i)\leq E$ for all $i$;
			\item for each $i$ and 
			each $W\in\mathfrak S$, if $V_i\propnest W$ or $W\trans V_i$, we have 
			$d_{W}(\pi_W(x),\rho^{V_i}_W)\leq E$.
		\end{itemize}
	\end{enumerate}

	We use $\mf{S}$ to denote the hierarchically hyperbolic space structure, including the index set $\mf{S}$, spaces $\{\mc{C}W : W \in \mf{S}\}$, projections $\{\pi_W : W \in \mf{S}\}$, and relations $\nest$, $\perp$, $\trans$. We call the elements of $\mf{S}$ the \emph{domains} of $\mf{S}$ and call the $\rho_W^V$ the \emph{relative projection} from $V$ to $W$. The number $E$ is called the \emph{hierarchy constant} for $\mf{S}$.
	
	We call a quasigeodesic space $\mc{X}$ a  \emph{hierarchically 
	hyperbolic space with constant $E$} if there exists a 
	hierarchically hyperbolic structure on $\mc{X}$ with constant $E$. We use the pair $(\mc{X},\mf{S})$ to denote a  hierarchically hyperbolic space equipped with the specific  HHS structure $\mf{S}$.
\end{defn}	

When writing the distances in the hyperbolic spaces $\mc{C}W$ between images of points under $\pi_W$, we will frequently suppress the $\pi_W$ notation. That is, we will use $d_W(x,y)$ to denote $d_{W}(\pi_W(x),\pi_W(y))$ for $x,y\in\mc{X}$. 

For a hierarchically hyperbolic space $(\mc{X},\mf{S})$, we are often most concerned with the domains $W \in \mf{S}$ whose associated hyperbolic spaces $\mc{C}W$ have infinite diameter. We often also restrict to HHSs with the following regularity condition.

\begin{defn}[Bounded domain dichotomy]
	Given an HHS $(\mc{X},\mf{S})$, we let $\mf{S}^\infty$ denote the set $\{W \in\mf{S}: \diam(\mc{C}W) = \infty\}$. We say that $(\mc{X},\mf{S})$ has the \emph{bounded domain dichotomy} if there is some $D \geq 0$ so that for all $W\in\mf{S} - \mf{S}^{\infty}$ we have  $\diam(\mc{C}W) \leq D$.
\end{defn}

The bounded domain dichotomy is a natural condition as it is satisfied
by all \emph{hierarchically hyperbolic groups (HHG)} which is a condition 
requiring equivariance of the HHS structure.  In this paper we work 
with a class of finitely generated groups which is slightly more 
general than being an HHG (see Remark~\ref{HHGvsNHHG}); these are groups which have an HHS structure compatible
with the action of the group in the following way.

\begin{defn}[$G$--HHS]\label{defn:nearlyHHG}
	Let $G$ be a finitely generated group.  A hierarchically
	hyperbolic space $(\mc{X},\mf{S})$ with constant $E$ that has the bounded domain dichotomy is a 
	\emph{$G$--HHS} if the following hold.
	\begin{enumerate}
		\item $\mc{X}$ is a proper metric space with a proper and cocompact action of $G$ by isometries.
		\item  $G$ acts on $\mf{S}$ by  $\nest$--, $\perp$--, and $\trans$--preserving bijections, and $\mf S^\infty$ has finitely many $G$--orbits.
		\item For each $W \in \mf{S}$ and $g\in G$, there exists an isometry $g_W \colon \mc{C}W \rightarrow \mc{C}gW$ satisfying the following for all $V,W \in \mf{S}$ and $g,h \in G$.
		\begin{itemize}
			\item The map $(gh)_W \colon \mc{C}W  \to \mc{C}ghW$ is equal to the map $g_{hW} \circ h_W \colon \mc{C}W \to \mc{C}ghW$.
			\item For each $x \in \mc{X}$, $g_W(\pi_W(x)) \asymp_E \pi_{gW}(g \cdot x)$.
			\item If $V \trans W$ or $V \propnest W$, then $g_W(\rho_W^V)  \asymp_E \rho_{gW}^{gV}$.
		\end{itemize}
	\end{enumerate}
We can and will assume that $\mc X$ is $G$ equipped with a finitely 
generated word metric. We say that $\mf{S}$ is a \emph{$G$--HHS structure} for  the group $G$ and use the pair $(G,\mf{S})$ the group $G$ equipped with the specific $G$--HHS structure $\mf{S}$.
\end{defn} 

\begin{rem}[HHG vs $G$--HHS]\label{HHGvsNHHG}
	The only difference between the above definition of a $G$--HHS and a hierarchically hyperbolic
	group is that a hierarchically hyperbolic group is required to
	have finitely many orbits in $\mf{S}$ and not just
	$\mf{S}^\infty$. (Both HHGs and $G$--HHSs satisfy the bounded 
	domain dichotomy, but for HHGs this is a theorem and for $G$--HHSs it is by definition.) 
	
	The primary reason we choose to work with the
	slightly broader definition of a $G$--HHS  is that we are ultimately interested
	in the boundary defined by an HHS structure.  Since the definition
	of the boundary does not involve uniformly bounded diameter
	domains, the natural class of structures to think about are those
	that only have restrictions on the set of infinite diameter
	domains.
 	
	We note that \cite[Corollary~3.8]{ABD} states that applying the
	maximization procedure that we will introduce in Section
	\ref{sec:maximization} to an HHG results in an HHG. However, the
	argument in \cite{ABD} does not explicitly address the
	co-finiteness of the action on the added finite diameter domains
	(``dummy domains").  Thus that result only shows that the result
	is a $G$--HHS. Since that argument addresses the infinite
	diameter domains, it follows that applying the maximization
	procedure to a $G$--HHS again results in a $G$--HHS.
\end{rem}

A hallmark of hierarchically hyperbolic spaces is that every pair of points can be joined by a  special family of quasigeodesics called \emph{hierarchy paths}, each of which projects to a quasigeodesic in each of the spaces $\mc{C}W$.

\begin{defn}
	A \emph{$\lambda$--hierarchy path} $\gamma$ in an HHS $(\mc{X},\mf{S})$
	is a $(\lambda,\lambda)$--quasigeodesic with the property that $\pi_W\circ
	\gamma$ is an unparametrized $(\lambda,\lambda)$--quasigeodesic
	for each $W \in \mf{S}$.
\end{defn}

\begin{thm}[{\cite[Theorem~4.4]{BHS_HHSII}}]
	Let $(\mc{X},\mf{S})$ be an HHS with constant $E$. There exist $\lambda\geq 1$
	depending only on $E$ so that every pair of point
	points in $\mc{X}$ is joined by a $\lambda$--hierarchy
	path.
\end{thm}

In general, a quasigeodesic (or even geodesic) in an HHS can be arbitrarily far from being a hierarchy path. Moreover, a given space might have  different HHS structures, and the set of hierarchy paths with respect to each structure might be  different.

\subsection{Hierarchical quasiconvexity and standard product regions}

The analogue of quasiconvex subsets of a hyperbolic space in the
setting of hierarchical hyperbolicity are the following
\emph{hierarchically quasiconvex} subsets. We refer the reader to 
\cite[\S5]{BHS_HHSII} for details on any of the background material 
in this subsection.

\begin{defn}
	Let $k \colon [0,\infty) \to [0,\infty)$. A subset $\mc{Y}$ of an HHS $(\mc{X},\mf{S})$ is \emph{$k$--hierarchically quasiconvex} if 
	\begin{enumerate}
		\item for each $W \in \mf{S}$, $\pi_W(\mc{Y})$ is a $k(0)$--quasiconvex subset of $\mc{C}W$;
		\item if $x \in \mc{X}$ so that $d_W(x,\mc{Y}) \leq r$ for each $W\in\mf{S}$, then $d_{\mc{X}}(x,\mc{Y}) \leq k(r)$.
	\end{enumerate}
\end{defn}

As with hierarchy paths, whether or not a subset is hierarchically quasiconvex can depend on which HHS structure is put on the space, hence $\mc{Y}$ is a hierarchically quasiconvex subset of $(\mc{X},\mf{S})$ and not just $\mc{X}$. 

Each hierarchically quasiconvex subset $\mc Y$ comes equipped with a \emph{gate map} denoted $\gate_{\mc Y} \colon \mc{X} \to \mc{Y}$. While this map might not be the coarse closest point projection, it has a number of nice properties that we summarize below.

\begin{lem}[{\cite[Lemma 5.5]{BHS_HHSII} \text{plus} \cite[Lemma 1.20]{BHS_HHS_Quasiflats}}]\label{lem:hierarch_path_through_the_gate}
	Let $(\mc{X},\mf{S})$ be an HHS with constant $E$. Suppose $\mc{Y} \subseteq \mc{X}$ is $k$--hierarchically quasiconvex. There is a coarse map $\gate_{\mc{Y}}\colon \mc{X} \to \mc{Y}$ and a constant $\kappa \geq 1$ depending only on $k$ and $E$, so that the following hold.
	\begin{itemize}
		\item $\gate_{\mc{Y}}$ is coarsely the identity on $\mc{Y}$.
		\item $\gate_{\mc{Y}}$ is $(\kappa,\kappa)$--coarsely Lipschitz.
		\item For each $x \in \mc{X}$ and $W \in \mf{S}$ we have \[\pi_W(\gate_{\mc{Y}}(x)) \asymp_\kappa \mf{p}_{\pi_W(\mc{Y})}(\pi_W(x)).\]
		\item For each $x \in \mc{X}$ and $y \in \mc{Y}$, there is a
		$\kappa$--hierarchy path $\gamma$ from $x$ to $y$ with the 
		property that $\gamma \cap \mc{N}_\kappa \bigl( \gate_{\mc{Y}}(x) \bigr) \neq \emptyset$.
	\end{itemize}
\end{lem}

Associated to each domain $W\in \mf{S}$ of an HHS $(\mc X, \mf S)$ is
a pair of hierarchically quasiconvex subspaces $\F_W$ and
$\mathbf{E}_W$.  Since we will not work directly with the definition of
these subsets,  we will just state the salient properties that we
will need.  For details, we direct the interested reader to
\cite[\S 5B]{BHS_HHSII} for the definition and proofs of their basic
properties.  In the sequel, we shall work primarily with the $\F_W$,
but we include the companion facts about $\E_W$ for context.

\begin{prop}\label{prop:properties_of_F}
	Let $(\mc{X},\mf{S})$ be an HHS with constant $E$. For each $W\in \mf{S}$, the subsets $\F_W,\mathbf{E}_W \subseteq \mc{X}$ have the following properties.
	\begin{enumerate}
		\item \label{F_prop:HQC}  There exists $k \colon [0,\infty) \to [0,\infty)$ depending only on $E$ so that $\F_W$ and $\mathbf{E}_W$ are $k$--hierarchically quasiconvex.
		\item \label{F_prop:projections} There exists $\kappa \geq 0$ depending only on $E$ so that for each $V \in \mf{S}$:
		\begin{itemize}
			\item $\pi_V(\F_W) \asymp_\kappa \rho_W^V$ and $\pi_V(\mathbf{E}_W) \asymp_\kappa \rho_V^W$ when  $W \propnest V$ or $W \trans V$;
			\item $\mc{C}V = \mc{N}_\kappa(\pi_V(\F_W))$ and $\diam(\pi_V(\mathbf{E}_W)) \leq \kappa$ when $ V\nest W$; and
			\item $\diam(\pi_V(\F_W) ) \leq \kappa$ and $\mc{C} V = \mc{N}_\kappa(\pi_V(\mathbf{E}_W))$ when $V \perp W$.
		\end{itemize} 
	Moreover, if $V \nest W$, then $\pi_W(x) \asymp_\kappa \pi_W(\gate_{\F_W}(x))$ for each $x \in \mc{X}$.
		\item \label{F_prop:unbounded} If $(\mc{X},\mf{S})$ has the bounded domain dichotomy, then $\diam(\F_W) =\infty$  if and only if $\mf{S}_W \cap \mf{S}^\infty \neq \emptyset$. Similarly, $\diam(\E_W) = \infty$ if and only if $\mf{S}_V^\perp \cap \mf{S}^\infty \neq \emptyset$ in this case. 
	\end{enumerate}
\end{prop}

\begin{rem}\label{rem:properties_of_F}  The construction of $\F_W$ 
	and $\E_W$ in \cite{BHS_HHSII} involves some choices, but all choices will produce coarsely equal subsets that satisfy the above properties.
\end{rem}

While we will not use this structure directly, the $\F_W$ and $\E_W$ form natural product regions in $\mc{X}$ as follows: equipping each $\F_W$ and $\mathbf{E}_W$ with the metric induced from $\mc{X}$, there is a quasi-isometric embedding $\F_W \times \E_W \to \mc{X}$ that sends $\F_W \times \{e\}$ and $\{f\} \times \E_W$ onto $\F_W$ and $\E_W$ for some $e\in \E_W$ and $f \in \F_W$. The image of this quasi-isometric embedding is often called the \emph{standard product region} for $W$ and denoted $\P_W$.

\subsection{The boundary of a hierarchically hyperbolic space}
Durham, Hagen, and Sisto defined a boundary for an HHS
$(\mc{X},\mf{S})$ that is built from the boundaries of the hyperbolic
spaces in $\mf{S}$; \cite{HHS_Boundary} is the reference for this 
subsection.  

The points in the HHS boundary are organized in a simplicial
complex that we denote $\partial_\Delta(\mc{X},\mf{S})$.  The
vertex set of $\partial_\Delta(\mc{X},\mf{S})$ is the set of all
boundary points of all the hyperbolic spaces $\mc{C}U$ for $U \in
\mf{S}^{\infty}$.  That is, the set of vertices is the set of points 
 $\bigcup_{U\in \mf{S}^{\infty}} \partial \mc{C}U$.  The vertices
$p_1,\dots,p_n$ of $\partial_\Delta(\mc{X},\mf{S})$ will form an
$n$--simplex if each $p_i \in \partial \mc{C}U_i$ and $U_i \perp U_j$
for each $i \neq j$.  This means the set of points making up the HHS
boundary can equivalently be described as the set of all linear
combinations $\sum_{U\in\mf{U}} a_U p_U$ where
\begin{itemize}
	\item $\mf{U}$ is a pairwise orthogonal subset of $\mf{S}^{\infty}$,
	\item$p_U \in \partial \mc{C}U$  for each $U \in\mf{U}$,  and
	\item $\sum_{U\in\mf{U}} a_U = 1$ and each $a_U > 0$.
\end{itemize}

\begin{defn}\label{def:suppset}
For each $p \in \partial_\Delta (\mc{X},\mf{S})$, we define
$\supp(p)$, the \emph{support of $p$}, to be the pairwise orthogonal
set $\mf{U} \subseteq \mf{S}$ so that $p = \sum_{U\in\mf{U}} a_U p_U$.
Equivalently, the support of $p$ is the pairwise orthogonal set
$\mf{U} \subseteq \mf{S}$ so that the smallest dimensional simplex of
$\partial_\Delta (\mc{X},\mf{S})$ that contains $p$ has exactly one
vertex from $\partial \mc{C}U$ for each $U \in\mf{U}$
\end{defn}

 For clarity, we will often decorate the coefficients $a_U$ with the boundary point they are the coefficient for, that is, we write $p = \sum_{U \in \supp(p)} a_U^p p_U$ to emphasize $a_U^p$ are the coefficients for the point $p$.
 
Durham, Hagen, and Sisto equip the HHS boundary with a topology beyond that coming from the simplicial complex described above. The definition of this topology is  quite involved, combining the standard topology on the boundaries of the hyperbolic spaces $\mc{C}U$ with projections of boundary points onto certain domains of the HHS structure.  We will define these boundary projections in Section \ref{subsec:boundary_projections} and use these boundary projections to define the topology in  Section \ref{subsec:top_on_bounadry}. After defining the boundary, we will describe when certain maps that respect the  HHS structure extend to maps on the boundary  in Section \ref{subsec:boundary_maps}. When our HHS is in fact a $G$--HHS, this will produce a natural action of the group on the boundary by homeomorphisms and simplicial automorphisms. 

We use $\partial (\mc{X},\mf{S})$ to denote the HHS boundary equipped with the topology coming from the boundary projections while $\partial_\Delta(\mc{X},\mf{S})$ will denote the simplicial complex that is the underlying set of boundary points.

\subsubsection{Boundary projections}\label{subsec:boundary_projections}
The following is a slight modification of the definition of a boundary projection from \cite{HHS_Boundary}.  
\begin{defn}\label{def:boundaryprojections}
Fix a point $q = \sum_{W \in \supp(q)}a_W q_W \in \partial (\mc X,\s)$.  For each $U\in \s$ such that there exists $W\in \supp(q)$ with $U\not\perp W$, we define the \emph{boundary projection $\partial \pi_U(q)$} of $q$ into $\mc CU$ as follows.
	\begin{itemize}
	\item If $W = U$, then we define $\partial \pi_U(q) := q_U=q_W$.
	\item If $W\propnest U$ or $W\trans U$,  let $\mc{V} = \{V \in \supp(q) \mid  V \trans U \text{ or } V \propnest U \}$, then define $$\partial \pi_U(q): = \bigcup_{V \in \mc{V}} \rho^V_U.$$
	\item If $W\sqsupsetneq U$, then $U \perp V$ for each $V \in \supp(q)-\{W\}$. In this case,  let $Z\subseteq \mc C W$ be the set of  all points on all $(1,20E)$--quasigeodesics from a point in $\rho^U_W\in\mc CW$ to $q_W\in\partial \mc CW$ that are at distance at least $E+\sigma$ from $\rho^U_W$, where $\sigma$ is the Morse  constant  from Lemma \ref{lem:Morse} for a $(1,20E)$--quasigeodesics in a $E$--hyperbolic metric space.  We then define $\partial \pi_U(q): =\rho^W_U(Z)$.
	\end{itemize} 
\end{defn}

The difference between this definition and the original from \cite{HHS_Boundary} is that \cite{HHS_Boundary} only defines the boundary projection of $q$ to a certain set of domains related to the support set of another point in the boundary.  We define the boundary projection of $q$ to any domain for which the definition  makes sense. Because of this, our notation is different  from what is used in \cite{HHS_Boundary}: they use $(\partial \pi_{\overline S}(q))_U$, where $\overline S$ is the support set of some point in the boundary, while we use $\partial\pi_U(q)$.

\subsubsection{Topology on $\mc{X}\cup\partial (\mc{X},\s)$} \label{subsec:top_on_bounadry}

Before defining the topology, we define the notion of a remote point.  This is a slight modification of \cite[Definition~2.5]{HHS_Boundary}, where they define a point being remote to a support set.
\begin{defn}\label{def:remote}
Let $(\mc X,\mf{S})$ be a hierarchically hyperbolic space, and let $p\in \partial(\mc X,\s)$.  A point $q\in \partial (\mc X,\s)$ is \emph{remote to $p$} if:
	\begin{enumerate}
	\item $\supp(p)\cap \supp(q)=\emptyset$; and 
	\item for all $Q\in \supp(q)$, there exists $P\in \supp(p)$ so that $P$ and $Q$ are \emph{not} orthogonal.
	\end{enumerate}
Denote the set of points remote to $p$ by $\partial^{rem}_p (\mc X,\mf{S})$.
\end{defn}

We are now ready to define the topology on $\partial (\mc X,\frak 
S)$.  Fix a basepoint $x_0 \in \mc{X}$,  and, for each $W \in\mf{S}^{\infty}$, pick the basepoint for $\partial \mc{C}W$ to be a point in $\pi_W(x_0)$. Fix a point $p = \sum_{W \in \supp(p)} a_W^p p_W \in \partial(\mc X,\frak S)$. For each $r\geq 0$, each $\varepsilon>0$, and each $W\in \supp(p)$, let $M(r;p_W)$ be a neighborhood of $p_W$ in $\mc CW\cup \partial \mc CW$ as  in Definition~\ref{def:nbhdbasis}.   We first define three sets depending on $r$ and $\varepsilon$: the \emph{remote part}, the \emph{non-remote part}, and the \emph{interior part}. In what follows, if $U$ is \emph{not} in the support of a boundary point $q$, then $a_U^q =0$; if $U$ is in $\supp(q)$ then $a_U^q$ is the coefficient of $q$ in the domain $U$ so that $q = \sum_{U \in \supp(q)} a_U^q q_U$.
 
\begin{defn}\label{defn:remote_top}
Given any $q\in \partial(\mc X,\frak S)$, let $\mathcal S_q$ be the union of $\supp(p)$ and the set of domains $T\in\supp(p)^\perp$ such that there exists some $W_T\in\supp(q)$ with $T\not\perp W_T$. The \emph{remote part} $ \mc B^{rem}_{r,\varepsilon}(p) $ is the set of all points $q\in \partial^{rem}_p(\mc X,\mf{S})$ satisfying the following three conditions:
	\begin{enumerate}[(R1)]
	\item $\forall W\in\supp(p), \partial\pi_W(q) \subseteq M(r;p_W)$,
	\item $\displaystyle \forall W\in\mc S_q, V\in\supp(p), \left|\frac{d_W(x_0,\partial\pi_W(q))}{d_{V}(x_0,\partial\pi_{V}(q))}-\frac{a^p_W}{a^p_{V}}\right|<\varepsilon,$ and 
	\item $\sum_{T\in \supp(p)^\perp} a^q_T<\varepsilon$.
	\end{enumerate}
\end{defn}

\begin{defn}\label{defn:nonremote_top}
  The \emph{non-remote part} $\mc B^{non}_{r,\varepsilon}(p)$ is the set of all points $$q=\sum_{T \in \supp(q)}a^q_Tq_T\in\partial \mc X-\partial^{rem}_p(\mc X,\mf{S})$$ satisfying the following three conditions, where  $\mc{A}= \supp(p)\cap\supp(q)$. 
	\begin{enumerate}[(N1)]
	\item $ \forall T\in \mc{A}, q_T\in M(r;p_T) $,
	\item $\sum_{V\in \supp(q)-\mc{A}} a^q_V<\varepsilon,$
	\item $  \forall T\in \mc{A},  \left|a^q_T-a^p_T\right|<\varepsilon.$
	\end{enumerate}
\end{defn}

\begin{defn}\label{defn:internal_top}
Finally, the \emph{interior part} $\mc B^{int}_{r,\varepsilon}(p)$ is the set of all points $x\in \mc X$ satisfying the following three conditions:  
	\begin{enumerate}[(\textrm{I}1)]
	\item $\forall W \in \supp(p)$, $\pi_W(x) \subseteq M(r;p_W)$,
	\item $\forall W,V \in \supp(p)$,  $\displaystyle \left|\frac{a^p_W}{a^p_{V}}-\frac{d_W(x_0,x)}{d_{V}(x_0,x)}\right|<\varepsilon$, and 
	\item $\forall W \in \supp(T)$, $T \in \supp(p)^\perp$, $\displaystyle \frac{d_T(x_0,x)}{d_W(x_0,x)}<\varepsilon$.
	\end{enumerate}
\end{defn}

The disjoint union of these sets forms a basic set in the topology.

\begin{defn}\label{defn:basisset} For each $\varepsilon >0$ and each $r\geq 0$, a basic set in the topology on $\mc X\cup \partial (\mc X,\s)$ is the set $\mc B_{r,\varepsilon}(p)$ defined as follows: 
\[
\mc B_{r,\varepsilon}(p):=\mc B^{rem}_{r,\varepsilon}(p)\sqcup\mc B^{non}_{r,\varepsilon}(p)\sqcup \mc B^{int}_{r,\varepsilon}(p).
\]
\end{defn}

In \cite{HHS_Boundary}, the basic sets are defined slightly differently. We briefly describe their definition here; Lemma \ref{lem:sametop} shows that the two collections generate the same topology. Given $p\in \partial (\mc X,\frak S)$ and $W\in\supp(p)$, let $K_W$ be \emph{any} neighborhood of $p_W$ in $\mc CW\cup \partial\mc CW$.  Let $\mc N^{rem}_{\{K_W\},\varepsilon}(p)$, $\mc N^{non}_{\{K_W\},\varepsilon}(p)$, and $\mc N^{int}_{\{K_W\},\varepsilon}(p)$ be the sets of points satisfying (R1)--(R3), (N1)--(N3), and (I1)--(I3), respectively, using the neighborhoods $K_W$ in place of $M(r;p_W)$.  Then define $\mc N_{\{K_W\},\varepsilon}(p)$ to be the disjoint union of these three sets.

\begin{lem}\label{lem:sametop}
The collections of sets $\mc B=\{\mc B_{r,\varepsilon}(p)\mid p\in \partial (\mc X,\frak S), r\geq 0, \varepsilon\geq 0\}$ and $\mc N=\{\mc N_{\{K_W\},\varepsilon}(p)\mid p\in \partial (\mc X,\frak S), \varepsilon\geq 0, \textrm{ and $K_W$ is a neighborhood of $p_W$ when $W \in \supp(p)$}\}$ generate the same topology on $\mc X\cup \partial(\mc X,\frak S)$.
\end{lem}

\begin{proof}
Since $\mc B\subseteq \mc N$, we need only show that for every $p\in \partial(\mc X,\frak S)$ and  every collection $\{K_W\}_{W\in \supp(p)}$ of neighborhoods of $p_W$, there is an $r\geq 0$ such that $\mc B_{r,\varepsilon}(p)\subseteq \mc N_{\{K_W\},\varepsilon}(p)$.  For each $W\in \supp(p)$, there is some $r_W$ such that $M(r_W;p_W)\subseteq K_W$.  Since $\supp(p)$ consists of finitely many domains, the result follows by setting $r=\max\{r_W\mid W\in \supp(p)\}$.
\end{proof}

The following technical lemma will be useful in verifying when points in the boundary lie in a particular basic set.
\begin{lem}\label{lem:changing_quotients_a_bounded_amount}
	Let $A_n$ and $B_n$ be sequences of positive numbers so that $A_n \to \infty$, $B_n \to \infty$ and $\lim_{n \to \infty} \frac{A_n}{B_n} = L.$ If there exists $E\geq0$ so that $|A_n-C_n| \leq E$ and $|B_n - D_n| \leq E$, then $\lim\limits_{n \to \infty} \frac{C_n}{D_n} = L$.
\end{lem}

\begin{proof}
	Fix $\varepsilon >0$. There exists $s>0$ sufficiently large and $r >0$ sufficiently small so that \[L \cdot r + \frac{1}{s} + \frac{r}{s} <\varepsilon.\]
	Since $A_n$ and $B_n$ tend to $\infty$,  \[ \lim_{n\to \infty}\frac{1 + \frac{E}{A_n}}{1+\frac{E}{B_n}} = 1 .\]
	Therefore, for all sufficiently large $n$ we have \[	\frac{C_n}{D_n} \leq \frac{A_n+E}{B_n-E} \leq \frac{A_n}{B_n}(1+r)\] where $r$ is the number fixed above. Since $\frac{A_n}{B_n} \to L$ as $n \to \infty$, for every large enough $n$, we have \[\frac{A_n}{B_n}(1+r) \leq \left(L + \frac{1}{s}\right)(1+r) \leq L + L\cdot r + \frac{1}{s} + \frac{r}{s} \leq L+\varepsilon. \]
	Hence, we have $\frac{C_n}{D_n} \leq L+\varepsilon$. 
	
	A completely analogous argument beginning with the inequality 
	\[
	\frac{C_n}{D_n} \geq \frac{A_n-E}{B_n+E}
	\]
	gives the lower bound, completing the proof.
\end{proof}

The next lemma says that sequences that stay within uniformly bounded distance of each other in $\mc{X}$ converge to the same point in the boundary $\partial(\mc{X},\mf{S})$.
\begin{lem}\label{lem:bounded_difference_convergence}
	Let $(\mc{X},\mf{S})$ be an HHS. Let $(x_n)$ be a sequence of points in $\mc{X}$ that converges to  $p \in \partial (\mc{X},\mf{S})$. If $(y_n)$ is a  sequence in $\mc{X}$ with $ d_\mc{X}(x_n,y_n)$ uniformly bounded for all $n \in\mathbb{N}$, then $y_n$ also converges to $p$.
\end{lem}

\begin{proof}
	Let $R$ be the uniform bound on  $ d_\mc{X}(x_n,y_n)$ and $x_0$ be the basepoint of the HHS boundary $(\mc{X},\mf{S})$. The sequence $(y_n)$ will converge to $p = \sum a_W p_W \in \partial (\mc{X},\mf{S})$ if for each $r \geq 1$ and $\varepsilon \geq 0$, we have $y_n \in \mc{B}^{int}_{r,\varepsilon}(p)$ for all but finitely many $n$.  Since $x_n$  converges to $p$,   for each $r \geq 1$ and $\varepsilon \geq 0$, there exists $n_0 = n_0(r,\varepsilon)$ so that for all $n \geq n_0$, $x_n \in  \mc{B}^{int}_{r,\varepsilon}(p)$. Thus, \begin{itemize}
		\item for each $W \in \supp(p)$, $\pi_W(x_n) \in M(r;p_W)$;
		\item for each $W,V \in \supp(p)$, $\lim\limits_{n\to \infty}\frac{d_W(x_0,x_n)}{d_{V}(x_0,x_n)} =  \frac{a_W}{a_V}$; and 
		\item for each $W \in \supp(p)$ and $T \in \supp(p)^\perp$, $\lim\limits_{n\to \infty}\frac{d_T(x_0,x_n)}{d_W(x_0,x_n)}=0$.
	\end{itemize}
	
	Let $W \in \supp(p)$. Since $d_W(x_n,y_n) \leq ER+E$ and $d_W(x_0,x_n) \to \infty$ as $n \to \infty$, there must exist $n_1 \in \mathbb{N}$ so that $y_n \in  M(r;p_W)$ for all $n \geq n_1$. Since $d_V(x_n,y_n) \leq ER+E$ for every $V \in \mf{S}$, we have $|d_V(x_0,x_n)-d_V(x_0,y_n)| \leq ER+E$ for each $V \in \mf{S}$. Thus, Lemma \ref{lem:changing_quotients_a_bounded_amount} implies that 
	\begin{itemize}
		\item for each $W,V \in \supp(p)$, $\lim\limits_{n\to \infty}\frac{d_W(x_0,y_n)}{d_{V}(x_0,y_n)} =  \frac{a_W}{a_V}$; and 
		\item for each $W \in \supp(p)$ and $T \in \supp(p)^\perp$, $\lim\limits_{n\to \infty}\frac{d_T(x_0,y_n)}{d_W(x_0,y_n)}=0$.
	\end{itemize}
	Hence $y_n \in \mc{B}^{int}_{r,\varepsilon}(p)$ for all sufficiently large $n$.
\end{proof}

\subsubsection{Boundary maps induced by hieromorphisms and the group action on the boundary}\label{subsec:boundary_maps}

Given two hierarchically hyperbolic spaces, it is natural to wonder when maps between the spaces extend to maps between their respective HHS boundaries. A natural class of maps to consider for this question are the following \emph{hieromorphisms}.

\begin{defn}
	Let $(\mc{X},\mf{S})$ and $(\mc{Y},\mf{T})$ be HHSs and $\lambda \geq 1$. A $\lambda$--\emph{hieromorphism} from  $(\mc{X},\mf{S})$ to $(\mc{Y},\mf{T})$ consists of 
	\begin{itemize}
		\item a map $f \colon \mc{X} \to \mc{Y}$;
		\item an injective map $f^\mf{s} \colon \mf{S} \to \mf{T}$ that preserves nesting, transversality, and orthogonality; and
		\item a $(\lambda,\lambda)$--quasi-isometric embedding $f_V \colon \mc{C}_\mf{S}V \to \mc{C}_\mf{T} f^{\mf{s}}(V)$ for each $V \in \mf{S}$	
	\end{itemize}
satisfying the following properties:
	\begin{itemize}
		\item $f_V(\pi_V(x)) \asymp_\lambda \pi_{f^{\mf{s}}(V)}(f(x))$ for each $x \in \mc{X}$ and $V \in \mf{S}$;
		\item $f_V(\rho_V^W) \asymp_\lambda \rho_{f^{\mf{s}}(V)}^{f^{\mf{s}}(W)}$ whenever $W \trans V$ or $W \propnest V$; and
		\item whenever $W\propnest V$, $f_W(\rho_W^V(z)) \asymp_\lambda \rho_{f^{\mf{s}}(W)}^{f^{\mf{s}}(V)}(f_W(z))$ for each $z \in \mc{C}V  -\mc{N}_E(\rho_V^W)$.
	\end{itemize}
	Since $f$, $f^{\mf{s}}$, and $f_V$ all have different domains, is it often clear from context which is the relevant map. In these cases, we will abuse notation and  call all maps $f$; we denote the hieromorphism by $f \colon (\mc{X},\mf{S}) \to (\mc{Y},\mf{T})$.
\end{defn}

Give a hieromorphism $f \colon (\mc{X},\mf{S}) \to (\mc{Y},\mf{T})$, each quasi-isometric embedding $f_V \colon \mc{C}V \to \mc{C}f(V)$ induces a continuous inclusion $\partial f_V \colon  \partial\mc{C}V \to \partial\mc{C}f(V)$. Since $f$ respects the orthogonality relation on the index sets, there is  an induced injective simplicial map $\partial f \colon \partial_\Delta (\mc{X},\mf{S}) \to \partial_\Delta(\mc{X},\mf{T})$ defined by $$\partial f(p) = \partial f \left(\sum_{V \in \supp(p)} a_V p_V\right) = \sum_{V \in \supp(p)} a_V\partial f_V(p_V).$$

Work of Durham, Hagen, and Sisto implies the following sufficient conditions for  this map $\partial f$ to be continuous with respect to the topology on the HHS boundary.

\begin{thm}[{A special case of  \cite[Theorem 5.6]{HHS_Boundary}}]\label{thm:extending_hieromorphisms}
	Let $f \colon (\mc{X},\mf{S}) \to (\mc{Y}, \mf{T})$ be a  hieromorphism.  If for each $V \in \mf{S}$,
	$f_V$ is a $(1,\lambda)$--quasi-isometric embedding, then the map $\partial f \colon \partial_\Delta (\mc{X},\mf{S}) \to \partial_\Delta(\mc{Y},\mf{T})$ defines a continuous map from $\partial (\mc{X},\mf{S})$ to $\partial (\mc{Y},\mf{T})$.
\end{thm}

\noindent We will use Theorem \ref{thm:extending_hieromorphisms} only in the following special case.

\begin{cor}\label{cor:extending_hieromorphisms}
	Let $(\mc{X},\mf{S})$ be an HHS. If $\mf{T}$  is another HHS structure for $\mc{X}$ and there is a hieromorphism $f \colon (\mc{X},\mf{S}) \to (\mc{X},\mf{T})$ so that $f \colon \mf{S} \to \mf{T}$ is a bijection and  $f_V$ is a $(1,\lambda)$--quasi-isometry for each $V \in \mf{S}$, then $\partial f \colon \partial (\mc{X},\mf{S}) \to \partial(\mc{X},\mf{T}) $ is a homeomorphism.
\end{cor}

The definition of a $G$--HHS ensures that the action of an element $g \in
G$ on $G$ by left multiplication gives a hieromorphism $g \colon
(G,\mf{S}) \to (G,\mf{S})$ where for each $V \in \mf{S}$, the map $g_V
\colon \mc{C}V \to \mc{C}gV$ is an isometry.  Thus, we can use
Corollary \ref{cor:extending_hieromorphisms} to extend the action of
$G$ on itself to an action of $G$ on $\partial (G,\mf{S})$ that is
both a homeomorphism and a simplicial automorphism.

\subsection{Hyperbolic HHSs}

A hyperbolic space can itself have many different HHS structures \cite{Spriano_hyperbolic_I}, but being hyperbolic puts a number of restriction on all of these HHS structures. The following result summarize the facts about hyperbolic HHSs that we will need.

\begin{thm}\label{thm:hyperbolic_HHS}
	Let $(\mc{X},\mf{S})$ be an HHS with constant $E$, and suppose $\mc{X}$ is also an $E$--hyperbolic space. 
	\begin{enumerate}
		\item For all $W \in \mf{S}^\infty$, we have $\mf{S}_W^\perp \cap \mf{S}^\infty = \emptyset$  \cite[Lemma 4.1]{HHS_Boundary}. In particular, the simplicial HHS boundary $\partial_\Delta(\mc{X},\mf{S})$ is  a collection of $0$--simplices.
		\item For each $\ell \geq 1$ and $c \geq 0$, there exists $\lambda \geq 1$ (depending on $\ell$, $c$,  and $E$) so that every $(\ell,c)$--quasigeodesic in $\mc{X}$ is a $\lambda$--hierarchy path in $(\mc{X},\mf{S})$ \cite[Proposition 3.5]{Spriano_hyperbolic_I}.
		\item The identity map $\mc{X} \to \mc{X}$ continuously extends to a homeomorphism from the Gromov boundary $\partial \mc{X}$ to the HHS boundary $\partial(\mc{X},\mf{S})$ \cite[Theorem 4.3]{HHS_Boundary}. 
	\end{enumerate}
\end{thm}

The next lemma describes how basis neighborhoods in $\mc{C}W \cup \partial \mc{C}W$ are related to basis neighborhoods in $\mc{X} \cup \partial \mc{X}$ when $(\mc{X},\mf{S})$ is a hyperbolic HHS. In the statement, we use $M(r,p)$ to denote basis neighborhoods in $\mc{X} \cup \partial \mc{X}$ and $M_W(r,p)$ to denote basis neighborhoods in $\mc{C}W \cup \partial \mc{C}W$.

\begin{lem}\label{lem:basisnbhd}
	Let $(\mc{X},\mf{S})$ be an HHS with constant $E$, and suppose $\mc{X}$ is also a $E$--hyperbolic space. There exists $r_0 \geq0$ so that for each $r \geq r_0$ and each $W \in \mf{S}^\infty$, there exists  $r' \geq 0$  so that the following hold.
	\begin{enumerate}
		\item\label{HypHHS:increasing} $r'$ is an increasing linear function of $r$ with constant of linearity determined by $E$;
		\item\label{HypHHS:moving_up} for each $p \in \partial \mc{C}W \subseteq \partial \mc{X}$ and each $x\in \mc{X}$, if $\pi_W(x) \subseteq M_W(r,p)$, then $x \in M(r',p)$; and
		\item\label{HypHHS:moving_sideways}  for each $p \in \partial \mc{C}W \subseteq \partial \mc{X}$ and each $q \in \partial \mc{C}W$, if $q \in M_W(r,p)$, then $q \in M(r';p)$.
	\end{enumerate} 
\end{lem}

\begin{proof}
Fix $W \in \mf{S}^\infty$ and a basepoint $x_0 \in \mc{X}$.  For the 
reader's 
convenience we will use $\gp{\cdot}{\cdot}$ to denote the Gromov
product in $\mc{X}$ and $\gph{\cdot}{\cdot}$ to denote the Gromov
product in $\mc{C}W$.  We first prove the second bullet point; 
the proof will determine the value of $r_0$ and calculate $r'$  in terms of $r$, which will establish the first bullet point.

 Fix $x \in \mc{X}$ and $p \in \partial \mc{C}W$. Let $\alpha$
be a $(1,20E)$--quasigeodesic ray from $x_0$ to $p$ in $\mc{X}$.  Since $\mc{X}$ is hyperbolic, Theorem~\ref{thm:hyperbolic_HHS} (2) provides 
$\lambda \geq 1$ so that each
$(1,20E)$--quasigeodesic (in particular $\alpha$) in $\mc{X}$ is a $\lambda$--hierarchy path
in $(\mc{X},\mf{S})$. 
Since the projections of hierarchy paths are unparametrized 
quasigeodesics, the projection $\pi_W \circ \alpha$ is an unparametrized
$(\lambda,\lambda)$--quasigeodesic in $\mc{C}W$. Thus Lemma
\ref{lem:Gromov_product_along_qg} provides a constant $B\ge
0$ that depends only on $E$ and a point $y \in \alpha$ so that \[|
\gp{x}{p} - \gp{x}{y}| \leq B \text{ and } |\gph{x}{p} - \gph{x'}{y'}|
\leq B,\] where $x'$ and $y'$ are any point in $\pi_W(x)$ and $\pi_W(y)$ respectively.  Let $\beta$ be a 
$(1,20E)$--quasigeodesic from $y$ to $x$ in $\mc{X}$, and consider the unparametrized
$(\lambda,\lambda)$--quasigeodesic
$\beta_W = \pi_W \circ \beta$  in $\mc{C}W$.  By the Morse Lemma
(Lemma \ref{lem:Morse}), $\beta$ (resp.  $\beta_W$) and any
$\mc{X}$--geodesic (resp.  $\mc{C}W$--geodesic) from $x$ to $y$ (resp. $x'$ to
$y'$) are each contained in the $\sigma$--neighborhood of each other
for some $\sigma$ determined by $E$.  Combining this with Lemma
\ref{lem:Gromov_product_is_distance_to_geodesic} yields \[|d_{\mc{X}}
(x_0, \beta) - \gp{x}{y}| \leq E+\sigma \text{ and } |d_{W}
(x_0,\beta_W) - \gph{x'}{y'}| \leq E+\sigma.  \] Since
$d_W(x_0,\beta_W) \leq E d_{\mc X}(x_0,\beta)+E$, we now have
\begin{align*}
				r \leq \gph{x'}{p} \leq & \gph{x'}{y'} +B \\
				\leq & d_{W}(x_0,\beta_W) +B+\sigma+E \\
			\leq & Ed_\mc{X}(x_0,\beta)+B+\sigma +2E\\
				\leq & E\gp{x}{y} + E(E+\sigma) +B+\sigma +2E \\
				\leq & E\gp{x}{p}+ EB + E(E+\sigma) +B+\sigma +2E. 
			\end{align*}
		Hence if $r'= \frac{1}{E}(r-B-\sigma)-B -\sigma -E - 2$ and  $r > (2B+2\sigma + E+2)E$, then $x \in M(r';p)$. 
		
		Now we establish the third bullet. Let $q \in \partial \mc{C}W \cap M_W(r;p)$. By Theorem \ref{thm:hyperbolic_HHS}, the inclusion map continuously extends to a homeomorphism between $\partial \mc{X}$ and $\partial(\mc{X},\mf{S})$. Hence, we can consider $q \in \partial \mc{X}$ and find a sequence $(q_n) \subseteq \mc{X}$ that converges to $q$ in both $\partial \mc{X}$ and $\partial(\mc{X},\mf{S})$. The definition of the topology on $\partial (\mc{X},\mf{S})$ ensures that $q_n \to q$ in $\mc{X}$ implies that for any choice $q'_n  \in \pi_W(q_n)$, $ q'_n \to q$ in $\mc{C}W$. Hence $\pi_W(q_n) \subseteq M_W(r;p)$ for all but finitely many $n$. The second bullet then says $q_n \in M(r';p)$ for all sufficiently large $n$. Hence $q \in M(r';p)$ as well.
\end{proof}

\section{Maximization and the boundary}\label{sec:maximization}
The goal of this section is to prove that the boundary of a proper
hierarchically hyperbolic space with the bounded domain dichotomy is
invariant under changing the structure by a procedure introduced in 
\cite{ABD} which we call
\emph{maximization}.   To simplify our arguments 
we break maximization into two steps. Given a hierarchically hyperbolic space $(\mc
X,\s)$ with the bounded domain dichotomy, we first replace $\s$ with the set of \emph{essential domains} (see
Section \ref{sec:maximizationsteps} for the definition), denoted
$\s_{ess}$.  Second, we apply the work of \cite{ABD} to obtain a new
hierarchical structure on $\mc X$, denoted $(\mc X,\frak T)$, that
satisfies several nice properties.  If we want to emphasize the
initial structure $\s$, we call this procedure \emph{maximizing the
structure $\frak S$}.  The resulting structure $\frak T$ is called the
\emph{maximized structure on $\mc X$ obtained from $\s$}, or simply
\emph{a maximized structure on $\mc X$}, if the structure $\s$ is
implicit or irrelevant.

Our main result is that the HHS boundary of a $G$--HHS  is invariant under maximization.

\begin{thm}\label{thm:BdryABDInvariant}
	If $(G,\mf{S})$ is a $G$--HHS and $\mf{T}$ is the maximized
	structure on $\mc X$ obtained from $\mf S$, then the identity map
	on $G$ extends to a $G$--equivariant map $ \partial(G,\mf{S}) \to
	\partial(G,\mf{T})$ that is both a simplicial isomorphism and a
	homeomorphism.
\end{thm}
 
 While the case of $G$--HHSs is likely of primary interest, our proof
 will not use a group action in any way and will apply to any
 hierarchically hyperbolic space that is proper and has the bounded
 domain dichotomy; see Theorem \ref{thm:ABD_Invariant_Spaces} for this
 more general statement.

\subsection{The maximization procedure}\label{sec:maximizationsteps}

In this section, we will provide a detailed description of the two steps
of maximization and show that the first step does not change the
boundary of a hierarchically hyperbolic space. The proof that the
boundary is invariant under the second step is more involved, and we 
will prove that in Section~\ref{sec:pfofBdryABDInvariant}, after
first developing some technical preliminaries in
Sections~\ref{sec:hqcinvar}--\ref{sec:bdryproj}.

Fix a hierarchically hyperbolic space $(\mc X,\s)$ with the bounded domain dichotomy. As every quasigeodesic space is quasi-isometric to a geodesic space, we will assume for convenience that $\mc{X}$ is a geodesic metric space.  In the context of $G$--HHSs,  the space $\mc X$ can be taken to be a Cayley graph of the group with respect to a finite generating set. 

\subsubsection*{Step 1: Essential domains} Let $\frak S_{ess}\subseteq \frak S$ be the set of domains $U\in \s$ such that there exists some $V\nest U$ so that $\mathcal C V$ has infinite diameter, that is, $V \in \mf{S}^\infty$. We call elements of $\s_{ess}$ \emph{essential domains}. The first step of maximization is to replace $\mf{S}$ with the set of essential domains $\mf{S}_{ess}$.

\begin{lem}\label{lem:removebounded}
Let $(\mc X,\s)$ be a hierarchically hyperbolic space with the bounded domain dichotomy.  Then $(\mc X,\frak S_{ess})$ is a hierarchically hyperbolic space and the identity $\mc{X} \to \mc{X}$ extends to map  $\partial (\mc X,\s_{ess}) \to \partial (\mc X,\frak S)$ that is both a simplicial automorphism and a homeomorphism.
\end{lem}

\begin{proof}
The set $\s-\s_{ess}$ is the set of domains $U\in \s$ such that
$\mathcal CV$ is uniformly bounded for every $V\nest U$.  Since this
set is clearly closed under nesting, it follows from \cite[Proposition
2.4]{BHS_HHS_AsDim} and the distance formula in hierarchically
hyperbolic spaces \cite[Theorem 4.5]{BHS_HHSII} that
$(\mc{X},\s_{ess})$ is a hierarchically hyperbolic space where all the
relations, hyperbolic spaces, and projections are the same as in
$(\mc{X},\mf{S})$.  This yields a hieromorphism $f\colon
(\mc{X},\s_{ess})\to (\mc{X},\s)$ where $f\colon \mc{X} \to \mc{X}$ is the identity, $f\colon \s_{ess}\to
\s$ is the inclusion, and $f_V$ is an isometry for all $V\in
\s_{ess}$.  Therefore by Theorem \ref{thm:extending_hieromorphisms},
there is an injective simplicial map $\partial f \colon
\partial_\Delta(\mc{X},\s_{ess})\to\partial_\Delta(\mc{X},\s)$ that is also  a continuous map $\partial
f\colon\partial (\mc{X},\s_{ess})\to \partial (\mc{X},\s)$.  Moreover,
since no domain in $\s-\s_{ess}$ contributes to $\partial
(\mc{X},\s)$, this map is a bijection and the basis neighborhoods
(given by Definition \ref{defn:basisset}) with respect to $\mf{S}$ and
$\mf{S}_{ess}$ will be identical.  Hence, the map $\partial f$ is 
a homeomorphism from $\partial (\mc{X},\s_{ess})\to \partial
(\mc{X},\s)$.
\end{proof}

We note that Lemma \ref{lem:removebounded} implies that if a group $G$ has two different $G$--HHS structures
$\s $ and $\s'$ such that $\mf{S}_{ess} =\mf{S'}_{ess}$, then
$\partial (G,\s)$ is homeomorphic to $\partial(G,\s')$.  More
generally, the two boundaries associated to $\s$ and $\s'$ are homeomorphic
if there exists a hieromorphism $(G,\s_{ess})\to(G,\s'_{ess})$ that
satisfies the condition of Corollary
\ref{cor:extending_hieromorphisms}.

\subsubsection*{Step 2: The new hierarchical structure}
We  describe the 
second and more involved step in the process of maximizing an HHS
structure $(\mc{X},\mf{S})$.  We refer the reader to \cite{ABD}
for the proof that this process in fact gives an HHS structure on $\mc{X}$.
We assume that we have already performed Step 1 so that $\mf{S}=\mf{S}_{ess}$.

Given $(\mc{X},\mf{S})$ an HHS with constant $E$ satisfying the bounded domain dichotomy, define $\mf{T}$ to be the subset of $\mf{S}$ containing the $\nest$--maximal element $S \in \mf{S}$ as well as all domains $W \in \mf{S}$ where $\F_W$ and $\E_W$ are both unbounded. Because  $\mf{S} =\mf{S}_{ess}$ and $\mf{S}$ has the bounded domain dichotomy, Proposition \ref{prop:properties_of_F} says $\F_W$ and $\E_W$ will both be unbounded if and only if $\mf{S}_W^\perp \neq \emptyset$. In particular, $W,V \in\mf{S}$ are orthogonal if and only if $W,V \in \mf{T}-\{S\}$ and are orthogonal in $\mf T$. 

The maximal structure on $\mc X$ obtained from $\s$ has index set $\mf T$. Before we
describe the full hierarchy structure associated to the set of domains
$\frak T$, we fix some notation to 
differentiate in which structure 
a domain is being considered.

\begin{notation}\label{not:twostructures} 
	To distinguish which structure we are working in ($\mf{S}$ vs 
	$\mf{T}$), we use the following convention. If nothing is 
	appended to the notation, it occurs in $(\mc X,\s)$; for example, 
	$\pi_W\colon \mc X\to \mc CW$ is the projection map in the 
	structure $(\mc X,\s)$.  For the hyperbolic spaces associated to 
	the structure $(\mc X,\frak T)$, we use the notation $\mc C_\frak 
	T W$ for each $W\in\frak T$.   For most other notation $*$ that 
	occurs in $(\mc X,\frak S)$, we will typically use $\overline{*}$ 
	to denote the corresponding object in $(\mc X,\frak T)$.  
	For example, $\overline\pi_W\colon \mc X\to \mc C_\frak T W$ is a 
	projection map in the structure $(\mc X,\frak T)$.  Similarly, 
	a point in $\partial (\mc X,\s)$ is simply denoted $p$, while
	a point in $\partial (\mc X,\frak T)$ is denoted $\ol p$. 
	Given a point $\ol p\in \partial (\mc X,\frak T)$, we denote its support in $\frak T$ by $\supp_\frak T(\ol p)$.
\end{notation}

The relations between domains in $\frak T$ are inherited from the relations in $\s$, i.e., the relation between $W,V \in \mf{T}$ is the same as the relation in $\mf{S}$.
 If $W\in  \frak T - \{S\}$, then $\mc CW = \mc C_\frak T W$ and the projection maps and relative projection maps are defined as in the original structure for  any $W\in \frak T-\{S\}$.
 
 Thus the only  associated hyperbolic space in the structure $(\mc X,\frak T)$ that is different is $\mc C_\frak TS$.   The hyperbolic space $\mc C_\frak TS$ is defined as follows.
 
 \begin{defn}
 	Let $\mc{C}_\mf{T}S$  be the space obtained from $\mc{X}$ by adding an edge of length 1 between every pair of points $x,y$ for which there is a $W \in \mf{T}-\{S\}$ so that $x,y \in \F_W$.
 \end{defn} 
 
  For the $\nest$--maximal domain $S\in \mf{T}$ and any $W\in \mf T-\{S\}$, we define $\ol\pi_S$
  to be the inclusion map $\mc X\to \mc C_\frak TS$, we define $\ol \rho^S_W$ 
  be the map $\ol\pi_W \circ \ol \pi_S^{-1}$, and we define $\ol\rho^W_S$ to be
  the subset $\ol \pi_S(\F_W)$ in $C_\mf{T}S$.
 
  \begin{rem}
	Technically, $(\mc X,\frak T)$ as described is not a
	hierarchically hyperbolic space, because it may not satisfy the
	containers axiom (Definition
	\ref{defn:HHS}(\ref{axiom:containers})).  In order to fix this
	problem, we actually define $\frak T$ to be the union of the set
	described above along with a collection of \emph{dummy domains},
	whose associated hyperbolic spaces are points.  These dummy
	domains essentially take the place of any containers that we may
	have removed when initially forming $\frak T$ from $\s$.  Since 
	to each dummy domain the associated hyperbolic space is defined 
	to be a point, these domains 
	do not contribute in any way to the HHS boundary, and hence we can
	ignore them in this paper.  We refer the reader to \cite{ABD} for
	a detailed description of how the dummy domains are incorporated
	into the full hierarchy structure on $(\mc X,\frak T)$.
 \end{rem}
  
  The fact that $\mc{C}_\mf{T} S$ is a hyperbolic space is a
  consequence of the \emph{factored space} construction in an HHS
  introduced in \cite[\S2]{BHS_HHS_AsDim}.  In addition to
  hyperbolicity, this construction yields that $\mc{C}_\mf{T}S$
  inherits an HHS structure as described in the next result.

  \begin{prop}
	[{\cite[Proposition 2.4]{BHS_HHS_AsDim} plus \cite[Corollary
	2.16]{BHS_HHS_Quasiflats}}]\label{prop:CTS_is_HHS} Given an HHS
	$(\mc{X},\mf{S})$, there exists $E'\geq 0$, depending only on the
	HHS constant $E$ of $(\mc{X},\mf{S})$, so that the space
	$\mc{C}_\mf{T}S$ is $E'$--hyperbolic and admits an HHS
	structure with constant $E'$ that has 
	index set $(\mf{S} -\mf{T}) \cup \{S\}$ and where the
	hyperbolic spaces, relations, and projections are all inherited
	from $(\mc{X},\mf{S})$.
  \end{prop}

The domains, hyperbolic spaces, and relative projections for the HHS $(\mc{C}_\mf{T} S,(\mf{S} -\mf{T}) \cup \{S\})$ are all identical to their counterparts from $\mf{S}$. The projection maps need a little more illumination. Recall, $\mc{C}_\mf{T} S$ is the space $\mc{X}$ with additional edges attached. If $x \in \mc{C}_\mf{T}S$ is also a point of $\mc{X}$, then for each $W \in (\mf{S} - \mf{T}) \cup \{S\}$, the projection $\pi_W(x)$ is the same as the projection to $\mc{C}W$ in $\mf{S}$. If instead $x$ is a point on an edge $e$ that is added to $\mc{X}$ to make  $\mc{C}_{\mf{T}}S$, then $\pi_W(x)$ is the union of the images of two end points of $e$ under $\pi_W$. 
 
 There are two important consequences of Proposition
 \ref{prop:CTS_is_HHS} that we will use repeatedly for the remainder of
 the section.  First, Theorem \ref{thm:hyperbolic_HHS} applies to the
 hyperbolic HHS $(\mc{C}_\mf{T} S,(\mf{S} -\mf{T}) \cup \{S\})$, so we can identify the
 Gromov boundary of $\mc{C}_\mf{T} S$ with points in the Gromov
 boundaries of $\mc{C}W$ for $W \in (\mf{S} -\mf{T}) \cup \{S\}$.
 Second, we can use Lemma \ref{lem:basisnbhd} to relate neighborhoods
 in $\mc{C}W \cup \partial \mc{C}W$ to neighborhoods in $
 \mc{C}_\mf{T} S \cup \partial \mc{C}_\mf{T}S$ when $W \in (\mf{S}
 -\mf{T}) \cup \{S\}$.

In order to prove Theorem~\ref{thm:BdryABDInvariant}, it remains to
show that this second step in the maximization procedure (replacing $\mf{S}$ with $\mf{T}$) does not
change the boundary of a $G$--HHS.  The proof
of this fact is involved, and  we spend the next several subsections
developing the necessary machinery and establishing a number of 
preliminary results.
Theorem~\ref{thm:BdryABDInvariant} is then proven in
Section~\ref{sec:pfofBdryABDInvariant}.

\subsection{Invariance of hierarchy paths and hierarchical quasiconvexity under maximization}\label{sec:hqcinvar}

As in the previous subsection, $(\mc{X},\mf{S})$ is an HHS with constant $E$ and the
bounded domain dichotomy, and $(\mc{X},\mf{T})$ is the HHS produced
after maximizing $\mf{S}$.  We denote the $\nest$--maximal domain in
both structures by $S$.  The goal of this subsection is establish that
hierarchy paths and hierarchical quasiconvexity do not change under
the maximization procedure.  These results are used to prove that
Step 2 of the maximization procedure does not change the boundary, but 
we expect they will be of broader interest as well.

We start by quoting a result that says hierarchy paths with respect $\mf{S}$ are also hierarchy paths with respect to $\mf{T}$. This was established by the first two authors and Durham during the introduction of the maximization procedure.

\begin{lem}[{\cite[Special case of Lemma 3.6]{ABD}}]\label{lem:hierarchy_paths_under_ABD}
	For each $\lambda \geq 1$ there exists $\lambda' \geq \lambda$ for which the following holds: if $\gamma$ is a $\lambda$--hierarchy path in $(\mc{X},\mf{S})$, then $\overline{\pi}_{S} \circ \gamma$ is an unparametrized $(\lambda',\lambda')$--quasigeodesic of $\mc{C}_{\mf{T}} S$. In particular, every $\lambda$--hierarchy path of $(\mc{X},\mf{S})$ is also a  $\lambda'$--hierarchy path of $(\mc{X},\mf{T})$.
\end{lem}

Next we establish the converse of Lemma \ref{lem:hierarchy_paths_under_ABD}, that hierarchy paths with respect to $\mf{T}$ are also hierarchy paths with respect to $\mf{S}$. This establishes that $(\mc{X},\mf{S})$ and $(\mc{X},\mf{T})$ have the same set of hierarchy paths (with possibly different constants).

\begin{lem}\label{lem:hierarchy_paths_in_T_are_hierarchy_paths_in_S}
	For each $\lambda \geq 1$ there exists $\lambda' \geq \lambda$ for which the following holds: if $\gamma$ is a $\lambda$--hierarchy path in $(\mc{X},\mf{T})$, then $\gamma$ is also a $\lambda'$--hierarchy path of $(\mc{X}, \mf{S})$.
\end{lem}

\begin{proof}
	Let $\gamma$ be a $\lambda$--hierarchy path in $(\mc{X},\mf{T}$).
	For each $W \in  \mf{T} - \{S\}$, the projection $\pi_W \circ \gamma$
	is an unparametrized $(\lambda,\lambda)$--quasigeodesic because
	$\pi_W = \overline{\pi}_W$.  Now assume $W = S$ or $W \in \mf{S} -
	\mf{T}$. By Proposition \ref{prop:CTS_is_HHS}, the space $\mc{C}_\mf{T}S$ is a hierarchically
	hyperbolic space with respect to $\{S\} \cup (\mf{S} - \mf{T})$,
	where the projection maps are the projection maps in the structure
	$\mf{S}$.  Because $\mc{C}_\mf{T}S$ is hyperbolic, every
	quasigeodesic in $\mc{C}_\mf{T}S$ is a hierarchy path in every
	HHS structure by Theorem \ref{thm:hyperbolic_HHS}.
	Thus $\pi_W \circ \overline{\pi}_S \circ \gamma$ is an
	unparametrized $(\lambda',\lambda')$-quasigeodesic in $\mc{C}W$.
	Because $\overline{\pi}_S$ is the inclusion map, we have that
	$\pi_W \circ \gamma$ is an unparametrized
	$(\lambda',\lambda')$--quasigeodesic in $\mc{C}W$.
\end{proof}

Using a result of the third author with Spriano and Tran, the above lemmas imply that the sets of hierarchically quasiconvex subsets of $\mc{X}$  with respect to $\mf{S}$ and $\mf{T}$ are the same.

\begin{prop}\label{prop:hqcequivalence}
	A subset $\mc{Y} \subseteq \mc{X}$ is hierarchically quasiconvex with respect to $\mf{S}$ if and only if it is hierarchically quasiconvex with respect to $\mf{T}$. Further, the function of hierarchical quasiconvexity in either $\mf{S}$ or $\mf{T}$ will determine the function in the other. 
\end{prop}

\begin{proof}
	Lemmas \ref{lem:hierarchy_paths_under_ABD} and
	\ref{lem:hierarchy_paths_in_T_are_hierarchy_paths_in_S} show that
	every hierarchy path of $(\mc{X},\mf{S})$ is a hierarchy path of
	$(\mc{X},\mf{T})$ and vice-versa.  The proposition then follows
	from \cite[Proposition 5.7]{RST_Quasiconvexity}, which states that a subset $\mc{Y}$ of an HHS is hierarchically quasiconvex if and only if there is a function $F\colon [1,\infty) \to [0,\infty)$ so that for every $\lambda \geq 1$, every $\lambda$--hierarchy path based on $\mc{Y}$ is contained in the $F(\lambda)$--neighborhood of $\mc{Y}$. The statement on the function of hierarchical quasiconvexity also follows from \cite[Proposition 5.7]{RST_Quasiconvexity}, which additionally shows that the function $k$  of hierarchical quasiconvexity and the function $F$ each determine the other.
\end{proof}

Proposition \ref{prop:hqcequivalence} is most relevant for us in the
case of the sets $\F_W$ in $(\mc{X},\mf{S})$ and $\overline{\F}_W$ in $(\mc{X},\mf{T})$. While  $\F_W$ might not equal $\overline \F_W$ even when $W \in \mf{S} \cap \mf{T}$,  they are each hierarchically quasiconvex with respect to their respective structures (Proposition \ref{prop:properties_of_F}\eqref{F_prop:HQC}). By Proposition \ref{prop:hqcequivalence}, there is thus some  $k$ depending only on
$(\mc{X},\mf{S})$ such that $\F_W$ and $\overline \F_W$ are each $k$--hierarchically quasiconvex 
with respect to both $\mf{S}$ and $\mf{T}$.
 In particular, the projection $\ol\pi_S(\F_W)$ is
a $k(0)$--quasiconvex subspace of the hyperbolic space $\mc C_{\frak
T}S$.

Since a subset $\mc{Y}$ is hierarchically quasiconvex with respect to
$\mf{S}$ if and only if it is hierarchically quasiconvex with respect
to $\mf{T}$, such a subset has a gate map with respect to each
structure.  We denote the gate map in $\mf{S}$ by $\gate_\mc{Y}$ and
the gate map in $\mf{T}$ by $\ol \gate_{\mc{Y}}$.  Our final lemma
says these two gate maps are coarsely the same.  The key step is
relating the the gate map in $\mf{S}$ to the 
closest point projection onto the image of a hierarchically
quasiconvex subset in $\mc{C}_\mf{T}S$.

\begin{lem}\label{lem:gate_in_S_is_cpp_in_T}
	Suppose $\mc{Y} \subseteq \mc{X}$ is $k$--hierarchically quasiconvex with respect to $\mf{S}$. There exists $C_1,C_2\geq 0$ depending on $k$ and $\mf{S}$ so that for all $x \in \mc{X}$  we have \[ \mf{p}_{\overline{\pi}_S(\mc{Y})}(\ol\pi_S (x))  \asymp_{C_1} \overline{\pi}_S(\gate_{\mc{Y}}(x)) \text{ and } \gate_{\mc{Y}}(x) \asymp_{C_2} \overline{\gate}_\mc{Y}(x).\] 
\end{lem}

\begin{proof} Let $E$ be the hierarchy constant for $\mf{S}$ and $\mf{T}$. Fix $x\in \mc{X}$, and  
	let $y$ be any point of $\mc{Y}$ satisfying $\overline{\pi}_S(y) \in
	\mf{p}_{\overline{\pi}_S(\mc{Y})}(\overline{\pi}_S(x))$.  Since   $\ol \pi_S(y), \ol\pi_S (\gate_\mc{Y}(x))$ and   $\gate_{\mc{Y}}(x),\overline{\gate}_\mc{Y}(x)$  all have diameter uniformly bounded in terms of $E$, the two coarse equalities will follow if we can bound the distances $\ol d_S(\ol \pi_S(y), \ol\pi_S (\gate_\mc{Y}(x)))$ and $d_\mc{X}(\gate_{\mc{Y}}(x),\overline{\gate}_\mc{Y}(x))$, respectively.
	
	By Lemma
	\ref{lem:hierarch_path_through_the_gate}, there exists  $\lambda \geq 1$ depending only on $k$ and $\mf{S}$ so that there is a
	$\lambda$--hierarchy path $\gamma$ in $(\mc{X},\mf{S})$ that connects
	$x$ and $y$ and passes within $\lambda$ of $\gate_{\mc{Y}}(x)$. Let $y'$ be a point on $\gamma$ with $d_\mc{X}(y',\gate_\mc{Y}(x)) \leq \lambda$. 
	
	  By Lemma	\ref{lem:hierarchy_paths_under_ABD}, $\gamma$ is also a hierarchy
	path in $(\mc{X},\mf{T})$, and so $\overline{\pi}_S \circ \gamma$ is
	an unparametrized $(\lambda',\lambda')$--quasigeodesic in
	$\mc{C}_{\mf{T}}S$ for some $\lambda'$ ultimately depending only
	on $\mf{S}$ and $k$. 
	By the Morse Lemma (Lemma \ref{lem:Morse}), there is  $\sigma \geq0$, depending ultimately only on $\mf{S}$ and $k$, so that $\ol\pi_S ( \gamma)$ is contained in the $\sigma$--neighborhood of any $\mc{C}_\mf{T}S$--geodesic  $[\ol \pi_S(x),\ol \pi_S(y)]$. 
	Since $y' \in \gamma$ and $\ol \pi_S$ is $1$--Lipschitz, we know $\ol \pi_S (\gate_\mc{Y}(x))$ is within $\sigma' = \lambda+\sigma$ of $[\ol \pi_S(x),\ol \pi_S(y)]$. Since $\overline{\pi}_S(y) \in
	\mf{p}_{\overline{\pi}_S(\mc{Y})}(\overline{\pi}_S(x))$, we have $\ol d_S(\ol \pi_S(y), \ol\pi_S (\gate_\mc{Y}(x))) \leq 2\sigma ' +1$, where $\sigma'$ depends only on $\mf{S}$ and $k$. This establishes the first coarse equality.
	
	Since $\mf{p}_{\overline{\pi}_S(\mc{Y})}(\overline{\pi}_S(x)) \asymp_\lambda \ol\pi_S(\gate_\mc{Y}(x))$, the uniform bound on $d_\mc{X}(\gate_{\mc{Y}}(x),\overline{\gate}_\mc{Y}(x))$ now follows from the uniqueness axiom in $\mf{T}$ (Definition \ref{defn:HHS}\eqref{axiom:uniqueness}), because
	$\pi_U(\gate_{\mc Y}(x))$ will be uniformly close to
	$\overline{\pi}_U(\overline{\gate}_{\mc{Y}}(x))$ for all $U \in
	\mf{T} -\{S\}$. 
\end{proof}

\subsection{A bijection from $\partial_\Delta(\mc{X},\mf{S})$ to $\partial_\Delta(\mc{X},\mf{T})$}\label{subsec:defn_phi}
In this section, we define a simplicial isomorphism $$\phi\colon
\partial_\Delta(\mc{X},\mf{S}) \to \partial_\Delta(\mc{X},\mf{T}).$$
In Section \ref{sec:pfofBdryABDInvariant}, we will prove that this
map is a homeomorphism from $\partial(\mc{X},\mf{S})$ to
$\partial(\mc{X},\mf{T}).$ By Lemma \ref{lem:removebounded}, we may
assume that the first step of the maximization procedure has already
been applied to $(\mc{X},\mf{S})$.  Thus we have a standing assumption
for the remainder of this section that $\mf{S} = \mf{S}_{ess}$.

 We first define $\phi$ for points $p \in \partial_\Delta(\mc{X},\mf{S})$ whose support is contained in $\mf{T} -\{S\}$.  Recall,  if $W \in \mf{T} - \{S\}$, then $\mc{C}W = \mc{C}_\mf{T} W$. Moreover,  because $\mf{S} = \mf{S}_{ess}$, we have $W,V \in\mf{S}$ are orthogonal if and only if $W,V \in \mf{T}-\{S\}$ and are orthogonal in $\mf T$. Thus,  each point $p \in \partial_\Delta(\mc{X},\mf{S})$ with $\supp(p) \subseteq \mf{T} -\{S\}$  is also a point in $\partial_\Delta(\mc{X},\mf{T})$ with the same support. For such points we define $\phi(p) =p$.

Now consider $p \in \partial_\Delta(\mc{X},\mf{S})$ with $\supp(p)
\not \subseteq \mf{T}- \{S\}$.  As supports are pairwise orthogonal
collections of domains and the only non-singleton sets of
orthogonal domains of $\mf{S}$ are
contained in $\mf{T}-\{S\}$, this implies $\supp(p) = \{P\}$ for some
$P \in \{S\} \cup (\mf{S}-\mf{T})$.  In this case, we define $\phi$
using the fact given by Proposition \ref{prop:CTS_is_HHS} that
$(\partial \mc{C}_\mf{T}S,(\mf{S} - \mf{T}) \cup\{S\})$ is a
hyperbolic HHS. By Theorem~\ref{thm:hyperbolic_HHS}, the identity map
on $\mc{C}_\mf{T}S$ extends to a homeomorphism from the Gromov
boundary $\partial \mc{C}_\mf{T}S$ to the HHS boundary $\partial
(\mc{C}_\mf{T}S,(\mf{S} - \mf{T}) \cup\{S\})$.  This homeomorphism
gives a bijection from $\{ p \in \partial \mc{C} W \mid W \in \{S\}
\cup (\mf{S} - \mf{T})\}$ to $\partial \mc{C}_\mf{T}S$.  Hence, if $p
\in \partial \mc C W$ for some $W \in \{S\} \cup (\mf{S} -\mf{T})$,
then $\phi(p)$ will be the image of $p$ under this identification.
	
For each $p \in \partial_\Delta(\mc{X},\mf{S})$, we will denote $\phi(p)$ by $\ol{p}$.  For each $P \in \mf{S}$, we also define a corresponding domain $\ol P \in \mf{T}$ by $\ol P = P$ if $P \in \mf{T}$ and $\ol P = S$  if $P \in \mf{S} - \mf{T}$. This definition ensures the following basic fact.

\begin{lem}
	 If $p \in \partial_\Delta(\mc{X},\mf{S})$ and $P \in \supp(p)$, then $\ol P \in \supp_\mf{T}(\ol p)$. Moreover, $|\supp(p)| = |\supp_\mf{T}(\ol p)|$.
\end{lem}
\begin{proof}
	As described in the preceding paragraphs, either $\supp(p)
	\subseteq \mf{T}-\{S\}$ or $\supp(p) \subseteq \{S\} \cup (\mf{S}
	- \mf{T})$. If $\supp(p) \subseteq \mf{T} -\{S\}$, then $\supp(p) =
	\supp_\mf{T}(\ol p)$ and $\ol P = P$ for each $P \in \supp(p)$.
	Since $(\mf{S} -\mf{T}) \cup \{S\}$ does not contain any 
	pairwise orthogonal
	domains, if $\supp(p) \subseteq (\mf{S}-\mf{T}) \cup \{S\}$, then
	$\supp(p) =\{P\}$ for some $P \in (\mf{S}-\mf{T}) \cup \{S\}$.  In
	this case $\supp_\mf{T}(\ol p) = \{S\}$ and $\ol P = S$.
\end{proof}

Note that if $p = \sum_{W \in \supp(p)} a_W p_W$, then $\ol p = \sum_{W \in \supp(p)} a_{\ol{W}} \ol p_{\ol W}$ where $a_W = a_{\ol W}$ and $\ol p_{\ol{W}} = \phi(p_W)$.

\subsection{Defining neighborhoods in $(\mc{X},\frak S)$ and $(\mc{X},\frak T)$}\label{sec:nbhds}

The key step in proving that Step 2 of the maximization procedure does
not change the boundary is to understand how basis neighborhoods in
$\partial (\mc{X},\mf{T})$ relate to those in $\partial
(\mc{X},\mf{S})$.  

In addition to assuming that $\mf{S} = \mf{S}_{ess}$, we make the
following standing assumption to simplify notation.  One 
consequence of 
this assumption is that the projection map associated to any $W$ now 
has codomain
$\mc{C}W$; this ensures that the preimage $\pi_W^{-1}(X)$ is well defined for a
subset $X \subseteq \mc{C}W$.

\begin{StandingAssumption}\label{assumption:projections_are_points}
	Given an HHS $(\mc{X},\mf{S})$,  we will assume that for each $W \in
	\mf{S}$ and $x \in \mc{X}$, $\pi_W(x)$ is a single point instead
	of a bounded diameter set.  This can always be done by replacing
	the image $\pi_W(x)$ with a choice of a single point in
	$\pi_W(x)$.  This modification gives a hieromorphism where the map
	on index sets is bijection and the maps between hyperbolic spaces
	are isometries.  Hence, Corollary
	\ref{cor:extending_hieromorphisms} ensures that this assumption
	does not affect the HHS boundary.  Note this assumption may
	increase the hierarchy constant from $E$ to $3E$.
\end{StandingAssumption}

We also fix the following constant for the remainder of the section.

\begin{notation}\label{notation:big_HHS_constant}
	 We first fix a constant $E$ larger than twice all the HHS constants for $(\mc{X},\s)$, $(\mc{X},\frak T)$, and $(\mc{C}_\mf{T} S,(\mf{S} - \mf{T}) \cup \{S\})$, as well as the hyperbolicity constant of $\mc{C}_\mf{T}S$. This includes making $E$ larger than the diameter of the boundary projections in each structure. Next, we define a constant $C$  to be
	 \[C = 2\kappa + 8E +2B +1,\] 
	 where $E$ is the  constant fixed above,  $\kappa$ is the maximum of the constants for the gate map and the $\F_W$'s from Lemma~\ref{lem:hierarch_path_through_the_gate} and Proposition~\ref{prop:properties_of_F}\eqref{F_prop:projections}, and $B$ is the constant from Lemma \ref{lem:enlarging_boundary_neighborhoods} for an $E$--hyperbolic space. (Essentially, $C$ is chosen large enough to accommodate any coarseness from the HHS properties and to apply Lemma \ref{lem:enlarging_boundary_neighborhoods}.) 
\end{notation}  

For convenience, we also define the following subsets of $\mc{C}_\mf{T}S$. 
\begin{defn}\label{defn:YA}
	For each $W\in \s - \mf{T}$, define
	$Y_W$ to be the subspace of $\mc C_\frak TS$ given by  
	\[Y_W:=\overline\pi_S(\F_W).\]
\end{defn}

\begin{rem}[Quasiconvexity and boundary of $Y_W$]\label{rem:boundary_YQ}
	Because $\F_W$ is hierarchically quasiconvex in
	$(\mc X,\s)$, it is also hierarchically quasiconvex in $(\mc X,\frak T)$ by Proposition \ref{prop:hqcequivalence}. Let $k$ be a function so  each $\F_W$ is $k$-hierarchically quasiconvex with respect to both $\mf{S}$ and $\mf{T}$.  By the definition of hierarchical quasiconvexity, the space
	$Y_W$ is  a $k(0)$--quasiconvex subspace of the hyperbolic space $\mc C_\frak TS$.
	
	If $q \in \partial(\mc{X},\mf{S})$ with $\supp(q) = \{Q\}$ for some $Q \in \mf{S} -\mf{T}$, then Proposition \ref{prop:properties_of_F}\eqref{F_prop:projections} ensures that there is a sequence of points $(x_n)$ in $\F_Q$ so that the sequence $(\pi_Q(x_n))$ converges to $q$ in $\mc{C}Q \cup \partial \mc{C}Q$.     Since $Q \in \mf{S} -\mf{T}$, the point $\ol q$ is in $\partial \mc{C}_\mf{T}S$ and the first two parts of Lemma \ref{lem:basisnbhd} ensure that the sequence $(\ol\pi_S(x_n))$ will then converge to $\ol q$ in $\mc{C}_\mf{T}S \cup \partial \mc{C}_\mf{T}S$. Hence $\ol q \in \partial Y_Q$.
\end{rem}

Fix a basepoint $x_0 \in \mc{X}$, and let $\pi_W(x_0)$ be the 
basepoint with respect to which the boundary $\partial \mc{C}W$ is 
constructed for each $W\in\s$.  Given a point
$p=\sum_{i=1}^n a_{U_i}p_{U_i}\in \partial (\mc{X},\frak S)$ and a
basic set $\mc{B}_{r,\varepsilon}(p)$ in the topology on $\partial
(\mc{X},\frak S)$, there is an associated collection of neighborhoods
$M(r;p_{U_i})$ of $p_{U_i}$ in $\mc C U_i$.  The goal of this
subsection is to define an associated collection of sets in the
hyperbolic spaces in the structure $\frak T$.  In Section
\ref{sec:pfofBdryABDInvariant}, we will discuss how this associated
collection of sets is related to a basic set in the topology on
$\partial (\mc{X},\frak T)$.
Given a neighborhood $M(r;p_W)$ of $p_W \in  \partial \mc CW$, we will define a corresponding neighborhood  $ M(R_r;\ol p_{\ol W})$ in  $\mc C_\frak T \ol W\cup\partial \mc C_\frak T\ol W$. In what follows, $M^\circ( \ast;p_W)$ denotes the set $M(\ast;p_W) \cap \mc{C}W$,  that is, $M^\circ( \ast;p_W)$ is the subset of the neighborhood that is  \emph{not} in the boundary of $\mc{C}W$.
	
 \begin{defn}[Neighborhoods in $\frak T$] \label{def:TnbhdsfromSnbhds}
Let $M(r;p_W)$ be a neighborhood in $\mc CW\cup \partial \mc CW$ of $p_W \in \partial \mc CW$ for some $W \in \mf{S}$. Let $E$ and $C$ be the constants from Notation \ref{notation:big_HHS_constant} and assume $r \geq r_0$, where $r_0$ is the constant from Lemma \ref{lem:basisnbhd} for $E$.
 
 First, we define an intermediate subset $\widetilde M(r;\ol p_{\ol W})$ as
 \[
  \widetilde M(r;\ol p_{\ol W}):=\mc N_C\left(\overline\pi_{\ol W}\left(\pi_W^{-1}(\mc N_{C}(M^\circ(r;p_W)))\right)\right).
  \]   
  
  \noindent We now use $\widetilde M(r;\ol p_{\ol W})$ to define a neighborhood of $\ol p_{\ol W}$ in $\mc{C}_\mf{T} \ol W$. Our choice of
  $C$ is large enough that Lemma \ref{lem:enlarging_boundary_neighborhoods} gives $\mc N_{C}(M^\circ(r;p_W))\subseteq M(r-2C;p_W)$.   
  If $\ol W=W$, then $\mc CW = \mc C_\frak T W$ and $p_W = \ol p_{\ol W}$. Thus, we have $$ \widetilde M(r;\ol p_{\ol W}) \subseteq \mc{N}_{2C} (M^\circ(r,p_W)) \subseteq M(r-4C;\ol p_{\ol W}).$$  
   If instead $\ol W=S$, then $\ol p_{\ol W}$ is a point in $\partial \mc{C}_\mf{T} S$. Since  $r\geq r_0$, we can therefore apply Lemma \ref{lem:basisnbhd}\eqref{HypHHS:moving_up} to the $E$--hyperbolic HHS $(\mc{C}_\mf{T} S,\{S\} \cup (\mf{S} -\mf{T}))$ to see that $$\widetilde M(r;\ol p_{\ol W}) \subseteq M (r'; \ol p_{\ol W})$$ for some $r'$ determined by $r$ and $E$. 
   Setting $R_r = \max \{r', r-4C\}$, the desired neighborhood is  $M(R_r;\ol p_{\ol W})$.
\end{defn}

The next lemma verifies that $R_r$  is an increasing function of $r$. 
\begin{lem}\label{lem:nbhdsinnbhds}
Given a neighborhood $M(r;p_W)$ in $\mc CW\cup \partial\mc CW$ (where 
$r \geq r_0$)  and its associated neighborhood $M(R_r;\ol p_{\ol W})$ 
in $\mc C_\frak T \ol W\cup \partial\mc C_\frak T \ol W$ as in 
Definition~\ref{def:TnbhdsfromSnbhds}, the quantity $R_r$ is an 
increasing linear function of $r$ with the constant of linearity determined by $E$.
\end{lem}

\begin{proof}
From Definition \ref{def:TnbhdsfromSnbhds},  $R_r = \max\{r', r-4C\}$. 
Since $C$ is determined by the hierarchy constant $E$, the result 
follows  from Lemma \ref{lem:basisnbhd}\eqref{HypHHS:increasing}, as
$r'$ is an increasing linear function of $r$ 
with the constant of linearity determined by $E$.  
\end{proof}

\subsection{How boundary projections behave when switching structures}
\label{sec:bdryproj}
We now prove three technical lemmas that let us  understand how the boundary projections change when switching from $\mf{S}$ to its maximization $\mf{T}$.  These  lemmas will be essential in the proof of Theorem~\ref{thm:BdryABDInvariant}. 

The first lemma describes a specific situation when the boundary projection changes by only a uniformly bounded amount.

\begin{lem}\label{lem:boundaryprojintersect}
	Let $q,p\in\partial(G,\s)$, and suppose $q$ is remote to $p$ and $\overline q$ is remote to $\overline p$. Suppose $\supp(p) = \supp_\mf{T}(\ol p) \neq \{S\}$ and $\supp(q) = \{Q\}$ for some $Q  \in \{S\} \cup (\mf{S} - \mf{T})$. If $W\in\supp(p)$ or $W \in \supp(p)^{\perp}$ with $W \not \perp Q$, then we have $$\diam_{\mc CW}\left(\partial \ol \pi_W(\ol q) \cup \partial\pi_W(q)\right)\leq C-2E.$$ 
\end{lem}

\begin{proof}
	Let $\sigma$ be the Morse constant (Lemma \ref{lem:Morse}) for a $(1,20E)$--quasigeodesic in an $E$--hyperbolic space. Since $\supp(q) = \{Q\} \subset  \{S\} \cup (\mf{S} - \mf{T})$, we have $\supp_\mf{T}(\ol q) = \{S\}$.  Thus for any $W\in\mf T$, the boundary projection $\partial \ol \pi_W(\ol q)$ is defined as $\ol \rho_W^S(Z)$ where $Z$ is the set of all points of $\mc{C}_\mf{T}S$ that  are at least $E+\sigma$ far from  $\ol \rho_S^W = \ol \pi_S(\F_W)$ and lie on a $(1,20E)$--quasigeodesic from a point in $\ol \rho_S^W$ to $\ol q$. 
	
	On the other hand, the boundary projection $\partial \pi_W(q)$ depends on the relation between $W$ and $Q$.
	The only way  $\supp(p)= \supp_\mf{T}(\ol p) \neq \{S\}$ is if $\supp(p) \subseteq \mf{T}-\{S\}$. Since the only orthogonality of $\mf{S}$ or $\mf{T}$ happens in $\mf{T} -\{S\}$, we also have $\supp(p)^\perp \subseteq \mf{T}-\{S\}$.  This means $Q \not \nest W$ as  $Q \not \in \mf{T} - \{S\}$. Similarly, $Q \not \perp W$ as the only orthogonality occurs among domains of $\mf{T}-\{S\}$.
    Hence, we must have $Q \trans W$ or $W \propnest Q$.  We consider each case separately, because the definition of the boundary projection $\partial \pi_P(q)$ depends on which relation holds. 
	
	If $Q\trans W$,  the boundary projection of $q$ to $W$ is defined as $\partial\pi_W(q)=\rho^Q_W$.
	Since $\supp(q)=\{Q\}$ and $Q\not\in \frak T$, consider the subspace $Y_Q=\ol \pi_S(\F_Q)$ of $\mc C_\frak T S$.  By Remark \ref{rem:boundary_YQ}, $Y_Q$ is $k(0)$--quasiconvex subset of $\mc{C}_\mf{T} S$ and  $\ol q\in \partial Y_Q$. Hence, Lemma \ref{lem:qgeonearY} provides a constant $A \geq 0$ and an  $(1,20E+2A)$--quasigeodesic from a point in $\ol\rho^W_S$ to $\ol q$ that eventually lies in $Y_Q$. Denote this quasigeodesic by $\alpha$.  By the Morse Lemma (Lemma \ref{lem:Morse}), $\alpha$ is contained in a $\sigma'$--neighborhood of any $(1,20E)$--quasigeodesic from $\ol \rho_S^W$ to $\ol q$, where $\sigma'$ is determined by $E$. In particular, by going sufficiently far along $\alpha$,  we can find a point $x \in \alpha \cap Y_Q$ and a point $y \in Z$ so that the $\mc{C}_\mf{T}S$--geodesic from $x$ to $y$ avoids $\mc{N}_{2E}(\ol \rho_S^W)$. Moreover, we can choose $x$ and $y$ so that they are points in $\mc{X}$ in addition to points in $\mc{C}_\mf{T} S$. The bounded geodesic image axiom (Definition \ref{defn:HHS}\eqref{axiom:bounded_geodesic_image}) in $\mf{T}$ now says $\diam_{\mc CW}(\ol \rho_W^S(x) \cup \ol \rho_W^S(y)) \leq E$.  As $y\in Z$,  it follows that $\partial\ol\pi_W(\ol q)\asymp_E \ol\rho^W_S(y)$.
	
	By Proposition \ref{prop:properties_of_F}, $\pi_W(\F_Q) \asymp_\kappa \rho_W^Q$.    Since $x \in \alpha \cap Y_Q$, we have $\ol{\rho}_W^S(x) = \pi_W (\ol\pi_S^{-1}(x)) \subseteq \pi_W(\F_Q).$ We have shown that \[\partial\ol\pi_W(\ol q)=\ol\rho^S_W(Z)\asymp_E\ol\rho^S_W(y)\] and \[\partial \pi_W(q)=\rho^Q_W\asymp_\kappa\pi_W(\mathbf F_Q)\asymp_E \ol\rho^S_W(x). \] 
	
	As we have also shown that $\diam_{\mc CW}(\ol \rho_W^S(x) \cup \ol \rho_W^S(y)) \leq E$, we conclude that $$\diam_{\mc CW}\left(\partial \ol \pi_W(\ol q) \cup \partial\pi _W(q)\right) \leq \diam_{\mc CW}(\ol \rho_W^S(x) \cup \ol \rho_W^S(y)) +  2E+\kappa \leq 3E + \kappa.$$
	The definition of $C$ ensures $3E +\kappa \leq C -2E$, finishing the proof in this case.

	If $W\propnest Q$, the  boundary projection $\partial \pi_W(q)$ is defined in terms of projections of quasigeodesic rays.  
	Since $Q  \not\in \mf{T} -\{S\}$, $\partial \mc{C} Q$ is a subset of $\partial \mc{C}_\mf{T} S$. Pick $z \in \F_W$, and let   $\alpha$ be a $(1,20E)$--quasigeodesic in $\mc C_\mf{T} S $  from $\ol \pi_S(z) \subseteq \ol \rho^W_S$ to $ \ol q\in\partial \mc C Q \subset \partial \mc{C}_\mf{T} S$.  By Theorem \ref{thm:hyperbolic_HHS}, the quasigeodesic $\alpha$ is a hierarchy path in the hyperbolic HHS $\mc C_\frak T S$.  Thus, $\pi_Q \circ \alpha$ is an unparametrized $(\lambda,\lambda)$--quasigeodesic in $\mc{C}Q$ for some $\lambda$ determined by $E$, where here $\pi_Q$ is the projection map in $(\mc C_\frak T S,(\frak S-\frak T) \cup \{S\})$. Since $\rho_Q^W \asymp_C \pi_Q(\F_W)$ by Proposition \ref{prop:properties_of_F}, $\pi_Q \circ \alpha$ gives a quasigeodesic ray in $\mc{C}Q$ that goes from $ \pi_Q(z) \in \mc{N}_C(\rho_Q^W)$ to $q$. Hence, $\pi_Q (\alpha)$ will be contained  in a uniform neighborhood of any $(1,20E)$--quasigeodesic from $\rho_Q^W$ to $q$. Therefore, there is  a point $x \in \alpha$ so that if $\alpha_0$ is the subray of $\alpha$ starting at $x$, then $\ol \rho_W^S(\alpha_0) \subseteq  \partial \ol\pi_W (\ol q)$ and $\rho_W^Q(\pi_Q  (\alpha_0)) \subseteq \mc{N}_E(\partial \pi_W(q))$ (this second inclusion follows from the bounded geodesic image axiom in $\mf{S}$). In particular, $d_Q(\rho_Q^W, \alpha_0)$ and $\ol d_S(\ol \rho_S^W,\alpha_0)$ are both strictly larger than $E$. See Figure \ref{fig:boundaryprojclaim} for a summary of the situation.
	
	\definecolor{FigGray}{gray}{.4}
	\begin{figure}[h]
		\begin{tikzpicture}
			\node at (0,0) 	{\includegraphics[scale=.5]{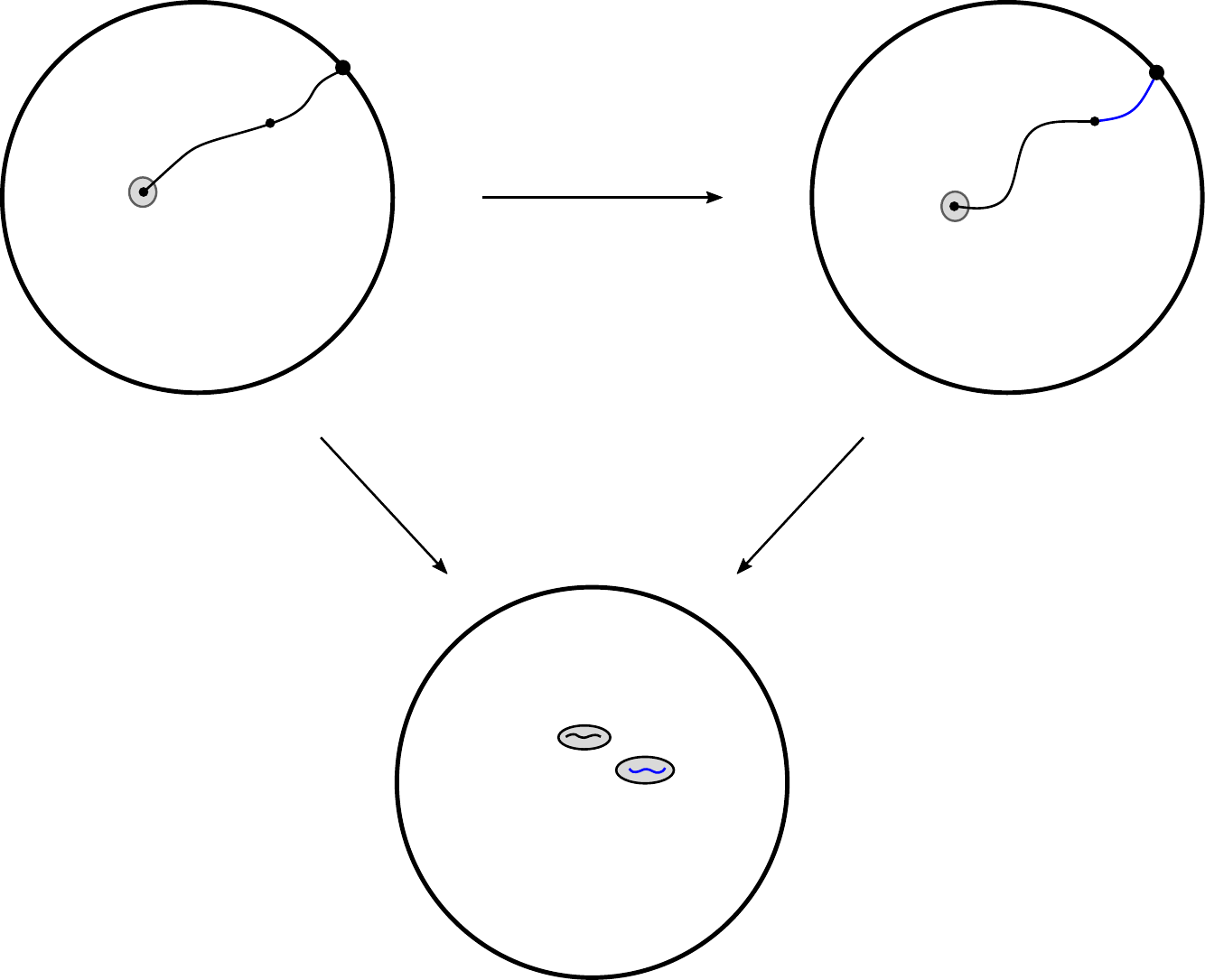}};
			\node at (-2.25,2.9) {\footnotesize $\mc{C}_\mf{T} S$};
			\node at (2.25,2.9) {\footnotesize $\mc{C}Q$};
			\node at (0, -3) {\footnotesize$\mc{C}W =  \mc{C}_\mf{T}W$};
			\node at (0,1.8) {\footnotesize $\pi_Q$};
			\node at (-1.45,-.3) {\footnotesize $\ol \rho_W^S$};
			\node at (1.3,-.3) {\footnotesize $\rho_W^Q$};
			\node at (-2.5,1.4) {\footnotesize \color{FigGray} $\ol \rho_S^P$};
			\node at (2.25, 1.2) {\footnotesize \color{FigGray} $\mc{N}_C(\rho_Q^W)$};
			\node at (-2.15, 1.8) {\footnotesize $\alpha$};
			\node at (-1.5, 2) {\footnotesize$\alpha_0$};
			\node at (-1.9,2.2) {\footnotesize $x$};
			\node at (-1.25, 2.55) { \footnotesize $\ol q$};
			\node at (4,1.9) {\footnotesize \color{blue} $\pi_Q(\alpha_0)$};
			\node at (3.2,2.55) {\footnotesize $q$};
			\draw [-stealth,blue,densely dotted] (3.45,1.9) -- (3,2.1);
			\node at (-.3,-1.1) {\footnotesize $\ol \rho_W^S(\alpha_0)$};
			\node at (.1,-1.85) {\footnotesize \color{blue} $\rho_W^Q(\pi_Q(\alpha_0))$};

		\end{tikzpicture}
		\caption{Proof of Lemma \ref{lem:boundaryprojintersect} when $W\propnest Q$.}
		\label{fig:boundaryprojclaim}
	\end{figure}

	Since $\ol \pi_S$ is the inclusion map, we can further select $x$ so that $\ol \pi_S^{-1} (x) = x$.  Since $d_Q(\rho_Q^W, x) >E$, the consistency axiom (Definition \ref{defn:HHS}\eqref{axiom:consistency}) in $\mf{S}$ says 
	\begin{equation}\label{eq:nesting_consistency}
		\diam_{\mc CW}\left(\rho^Q_W(\pi_Q(\ol\pi_S^{-1}( x))) \cup \pi_W(\ol\pi_S^{-1}( x))\right) \leq E. 
	\end{equation}
	By our choice of $x$, we have $\pi_Q(\ol\pi_S^{-1}(x))\subseteq \pi_Q(\alpha_0)$, and hence  $$\rho^Q_W(\pi_Q(\ol \pi_S^{-1}( x))) \subseteq \rho_Q^W(\pi_Q(\alpha_0)) \subseteq \mc{N}_E(\partial\pi_W(q)).$$ Equation \eqref{eq:nesting_consistency} then implies
	\begin{equation}\label{eqn:piP(q)}
		\diam_{\mc CW}\left(\partial\pi_W(q) \cup \pi_W(\ol\pi_S^{-1}(x))\right) \leq 4E.
	\end{equation}
	
	Similarly, because $\ol d_S(\ol \rho_S^W,x) >E$, the consistency axiom in $\mf{T}$ says
	
	\[ \diam_{\mc CW}\left(\ol\rho^S_W( x) \cup \pi_W(\ol\pi_S^{-1}(x))\right) \leq  E.
	\]
	Since $x$ was chosen so that $\ol \rho_W^S(x) \subseteq \partial \ol \pi_W(\ol q)$, this implies 
	\begin{equation}\label{eqn:piPqinT}
		\diam_{\mc CW}\left(\partial \ol \pi_W(\ol q) \cup \pi_W(\ol\pi_S^{-1}(\ol x))\right) \leq  2E.
	\end{equation}

	Applying the triangle inequality to  \eqref{eqn:piP(q)} and \eqref{eqn:piPqinT}, we obtain
	\begin{align*}
		\diam_{\mc CW}\left(\partial\pi_W(q) \cup \partial\ol\pi_W(\ol q)\right) & \leq 
		\diam_{\mc CW}\left(\partial \pi_W(q) \cup \pi_W(\ol\pi_S^{-1}(\ol x))\right) + \diam_{\mc CW}\left(\pi_W(\ol\pi_S^{-1}(\ol x))\cup \partial\ol\pi_W(\ol q)\right) \\ &
		\leq 4E + 2E =6E.
	\end{align*}
	As $6E < C- 2E$, this completes the proof of Lemma \ref{lem:boundaryprojintersect}. 
\end{proof}

The next two lemmas describe how switching structures affects the interaction of boundary projections  with neighborhoods.  Roughly, the lemmas state that if we have two points $p,q\in \partial (\mc{X},\frak S)$ with $q$ remote to $p$ and a domain $P\in\supp(p)$, then if the boundary projection in $(\mc{X},\frak S)$ of $q$ to $P$ is contained in the neighborhood $M(r;p_P)$ of $p_P$ in $\mc CP$, then the boundary projection in $(\mc{X},\frak T)$ of $\ol q$ to $\ol P$ (or $\ol q$ itself) is contained in the associated set $M(R_r;p_{\ol P})$ in $\mc C_\frak T \ol P$.  Here, $\ol P$ is as defined in Section \ref{sec:nbhds}; see Definitions~\ref{def:nbhdbasis} and \ref{def:TnbhdsfromSnbhds} for the definitions of $M(r;p_P)$ and $M(R_r;\ol p_{\ol p})$, respectively.  The statements are made  precise by considering how $\ol q$ and $\ol p$ are related.  Lemma~\ref{lem:bdryprojopensets} handles the case when  $\ol q$ is remote to $\ol p$ and is broken into two subcases depending on whether $\supp(p)=\supp_{\frak T}(\ol p)$ or not.  Lemma~\ref{lem:projopensetsnon} handles the case when  $\ol q$ is not remote to $\ol p$ and assumes that $\supp(p)\neq\supp_{\frak T}(\ol p)$, which is the only case we will need for the proof of Theorem~\ref{thm:BdryABDInvariant}.  

Recall, we are still operating under the standing assumptions that $\mf{S} = \mf{S}_{ess}$ (Lemma \ref{lem:removebounded}) and that $\pi_W(x)$ is a single point for each $x\in \mc{X}$ and $W \in \mf{S}$ (Standing Assumption \ref{assumption:projections_are_points}).

\begin{lem}\label{lem:bdryprojopensets}
Let $q,p\in\partial(G,\s)$, and suppose $q$ is remote to $p$ and $\overline q$ is remote to $\overline p$. Let $r \geq r_0$, where $r_0$ is the lower bound on $r$ required in Definition \ref{def:TnbhdsfromSnbhds}. 
\begin{enumerate}
\item \label{item:diffsupp} If $\supp(p)=\{P\}$ and $\supp_\frak T(\overline p)=\{S\}$ (including the possibility that $P= S$), then 
\[
\partial\pi_P(q) \subseteq M(r;p_P) \qquad \implies \qquad \partial\overline\pi_S(\overline q)\subseteq   M(R_r;\ol p_S).
\]
\item \label{item:samesupp} If $\supp(p)=\supp_\frak T(\overline p) \neq \{S\}$, then for any $P\in\supp(p)$
\[
\partial\pi_P(q)\subseteq  M(r;p_P) \qquad \implies \qquad \partial\overline\pi_P(\overline q)\subseteq   M(R_r;\ol p_P).
\]
\end{enumerate}
\end{lem}

\begin{proof}

\noindent{\bf Proof of (\ref{item:diffsupp}).} 
Suppose $\supp(p)=\{P\}$ and
$\supp_\frak T(\overline p)=\{S\}$.  We first determine how the boundary projections $\partial\ol \pi_S(\ol q)$ and $\partial\pi_P(q)$ are defined.  Since $\overline
q$ is remote to $\overline p$, we must have $\supp_\frak T(\ol p)\cap \supp_\frak T(\ol q)=\emptyset$ (see Definition~\ref{def:remote}), and so $S\not \in\supp_\frak T(\overline
q)$.  Thus $Q \propnest S$ for all $Q \in \supp_{\mf{T}}(\ol q)$, and the boundary projection (Definition~\ref{def:boundaryprojections}) of $\ol q$ to $S$ is defined as  $$\partial\ol \pi_S(\ol
q)= \bigcup_{Q \in \supp_\mf{T}(\ol q)} \ol \rho^Q_S.$$  Moreover, $\supp_\frak T(\ol q)\neq \{S\}$ implies that $\supp(q)=\supp_\frak T(\ol q)$, and so $Q\in\frak T -\{S\}$ for all $Q\in \supp(q)$.  However, since $P\not\in
\frak T - \{S\}$, it is not possible that $P\perp Q$ or $P\sqsubsetneq Q$ for
any $Q\in \supp(q)$.  Thus, for each $Q\in \supp(q)$, either
$Q\propnest P$ or $Q\trans P$.  In either case, the boundary projection of $q$ to $P$ is defined as $$\partial
\pi_P(q)= \bigcup_{Q \in \supp(q)} \rho^Q_P.$$

Fix $r\geq r_0$ where $r_0$ is the lower bound on $r$ require by Definition \ref{def:TnbhdsfromSnbhds}. Assume  $\partial\pi_P(q) \subseteq M(r;p_P)$ and let $Q \in \supp(q) = \supp_{\mf{T}}(\ol q)$. Proposition \ref{prop:properties_of_F}\eqref{F_prop:projections} says $\rho^Q_P \asymp_C \pi_P( \F_Q)$, which implies $$\F_Q\subseteq \pi_P^{-1} \left(\mc{N}_C(M^\circ(r;p_P)) \right).$$ However, $R_r$ was chosen so that this implies  $\ol \pi_S(\F_Q) = \ol \rho_S^Q \subseteq M(R_r, \ol p_{\ol P})$. Thus $\partial \ol \pi_S(q) \subseteq M(R_r, \ol p_{\ol P})$ as desired.

\medskip

\noindent{\bf Proof of \eqref{item:samesupp}.}
Suppose $\supp(p)=\supp_\frak T(\overline p) \neq \{S\}$. This only occurs when $\supp(p) \subseteq \mf{T} -\{S\}$.
 If $\supp_{\frak T}(\ol q) \neq \{S\}$, then the result is immediate because $\supp(q) = \supp_{\mf{T}}(\ol q) \subseteq \mf{T}-\{S\}$ and we have  $\partial \pi_P(q)=\partial\ol \pi_P(\ol q)$ and $\mc{N}_{2C}(M^\circ(r;p_P)) \subseteq M(R_r;\ol p_{\ol P})$; see Definition \ref{def:TnbhdsfromSnbhds}. So suppose $\supp_\frak T(\ol q) = \{S\}$.  This only occurs when  $\supp(q)=\{Q\}$ for some $Q \in \{S\} \cup (\mf{S} - \mf{T})$.  Since $Q \not \in \mf{T} -\{S\}$ but $\supp(p) \subseteq \mf{T}-\{S\}$, we know $Q \neq P$ and $Q \not \perp P$ for each $P \in \supp(p)$.

Because $\supp_\mf{T}( \ol p) = \supp(p) \neq \{S\}$, we have  $P\propnest S$ for all $P \in \supp(p)$. Let $\sigma$ be the Morse constant for a $(1,20E)$--quasigeodesic in an $E$--hyperbolic space and fix $P\in\supp(p)$.  Let $Z$ be the set of all points in $\mc{C}_\mf{T}S$ that are at least $E +\sigma$ far from $\ol \rho_S^P$ and are contained in a $(1,20E)$--quasigeodesic from a point in $\ol \rho^P_S$ to $\ol q_S=\ol q$. 
 The boundary projection of $\ol q$ to each $P \in \supp(p)$ is then defined as $\partial \ol \pi_P(\ol q)=\ol\rho^S_P(Z)$; see Definition \ref{def:boundaryprojections}.  

 Since $P$ is in both $\frak S$ and $\frak T$, we have  $\mc CP = \mc C_{\frak T}P$, and so we  consider $\partial \ol \pi_P(\ol q)$ as a subset of $\mc CP$.
Since $\partial\pi_P(q)$ has diameter at most $E$ and $\partial\pi_P(q)\subseteq M^\circ(r;p_P)$ by assumption, it follows from Lemma \ref{lem:boundaryprojintersect}   (with $W =P$)  that 
$\partial \ol \pi_P(\ol q)\subseteq \mc N_{C}(M^\circ(r;p_P)).$
Because $P = \ol P$,  the set $M(R_r,\ol p_P)$ is defined so that $\mc N_{C}( M(r;p_P))\subseteq M(R_r;\ol p_P)$, and the result follows.
\end{proof}

\begin{lem}\label{lem:projopensetsnon}
Let $p,q\in\partial(G,\s)$, and suppose $q$ is remote to $p$ but $\overline q$ is not remote to $\overline p$.   
Suppose $\supp(p)=\{P\}$ and $\supp_\frak T(\overline p)=\{S\}$ where $S\neq P$. If $r_0$ is the lower bound for $r$ from Definition \ref{def:TnbhdsfromSnbhds}, then for any $r>r_0$
\[
\partial\pi_P(q)\subseteq M(r;p_P)  \qquad \implies \qquad \overline q\in M(R_r; \ol p_S).
\]
\end{lem}

\begin{proof}
We have $\supp(p)=\{P\}$ and $\supp_\frak T(\overline p)=\{S\}$ where $S\neq P$. We first determine the supports of $q$ and $\ol q$.  Because there are no domains orthogonal to $S$, no domain in $\supp_\frak T(\ol q)$ is orthogonal to $S$.  Since we are assuming that $\overline q$ is not remote to $\overline p$, we therefore must have $\supp_\frak T(\overline q)\cap \supp_\frak T(\overline p)\neq\emptyset$; see Definition~\ref{def:remote}. 
It follows that $\supp_\frak T(\overline q)=\{S\}$, and so $\supp(q)=\{Q\}$ for some $Q\in\s$ with $Q \not \in \mf{T}  -\{S\}$. Since we are assuming $p$ and $q$ are remote, we also have $P \neq Q$.

Assume $\partial \pi_P(q)\subseteq M(r;p_P)$.  We consider two cases, depending on how the boundary projection $\partial \pi_P(q)$ is defined:  we first handle the case where either $Q\propnest P$ or $Q\trans P$, then we address the case where $Q\sqsupsetneq P$. 

\textbf{Case $Q \propnest P$ or $Q \trans P$:} When $Q\propnest P$ or $Q\trans P$, the boundary projection is defined as $\partial\pi_P(q)=\rho^Q_P$.  Since $\partial\pi_P(q) = \rho_P^Q\subseteq M(r;p_P)$ by assumption and $\rho_P^Q \asymp_C \pi_P(\F_Q$) by Proposition \ref{prop:properties_of_F}\eqref{F_prop:projections}, our choice of $R_r$ yields 
\begin{equation*}\label{eqn:rhoQP}
\ol\rho_S^Q =	\ol \pi_S(\F_Q)\subseteq   \ol \pi_S(\pi_P^{-1}( \mc N_C(M(r;p_P))) \subseteq M(R_r;\ol p_S)
\end{equation*}
 as desired.

\textbf{Case $Q\sqsupsetneq P$:}  Let $\sigma$ be the Morse constant (Lemma \ref{lem:Morse}) for a $(1,20 E)$--quasigeodesic in an $E$--hyperbolic space. Because $Q\sqsupsetneq P$, the boundary projection $\partial \pi_P(q)$ is defined as $\rho^Q_P(Z)$, where $Z$ is the collection of all points on $(1,20E)$--quasigeodesics in $\mc CQ$ from a point in $\rho^P_Q$ to $q$ that are at distance at most $E+\sigma$ from $\rho^P_Q$. 
Let $\alpha$ be a  $(1,20E)$--quasigeodesic ray $\alpha$ in $\mc{C}_\mf{T} S$ from a point in $\ol \pi_S(\F_P) = Y_P$ to $\ol q$. By Theorem \ref{thm:hyperbolic_HHS}, there is $\lambda \geq 1$ determined by $E$ so that $\alpha$ is $\lambda$--hierarchy path in the HHS $(\mc{C}_\mf{T} S,(\mf{S} -\mf{T}) \cup \{S\})$. In particular,  $\pi_Q \circ \alpha$ is an unparameterized $(\lambda,\lambda)$--quasigeodesic ray from a point in  $\mc{N}_E(\rho_Q^P)$ to $q$.  By the Morse Lemma (Lemma \ref{lem:Morse}), $\pi_Q(\alpha)$ is contained in a uniform neighborhood of any $(1,20E)$--quasigeodesic from a point in $\rho_Q^P$ to $q$. Hence, by going far enough along $\alpha$, we can find a subray $\alpha_0$ so that the consistency and bounded geodesic image axioms in $(\mc{C}_\mf{T} S, (\mf{S} -\mf{T})\cup \{S\})$ imply $\rho_P^Q(\pi_Q(\alpha_0)) \subseteq \mc{N}_E(\partial \pi_Q(q))$ and $\pi_P(\alpha_0) \asymp_E \rho_P^Q(\pi_Q(\alpha_0))$.
As a result, $$\pi_P(\alpha_0) \subseteq \mc{N}_{2E}(M(r;p_P)) \subseteq M(r -4E;p_P) \subseteq M(r-4C;p_P).$$
Lemma \ref{lem:basisnbhd}\eqref{HypHHS:moving_up} and our choice of $R_r$ (Definition \ref{def:TnbhdsfromSnbhds}) then imply $\alpha_0 \subseteq M(R_r;\ol p_{\ol {P}})$. Since $\alpha_0$ represents $\ol q$, this implies $\ol q \in M(R_r;\ol p_{\ol p})$ as desired.
\end{proof}

\subsection{Invariance of the boundary under maximization}\label{sec:pfofBdryABDInvariant}
We are now ready to prove that the maximization procedure does not
change the HHS boundary for proper HHSs with the bounded domain
dichotomy.  Since every finitely generated group is a proper metric
space and every $G$--HHS structure has the bounded domain dichotomy,
Theorem \ref{thm:BdryABDInvariant} is a special case of this result.

\begin{thm}\label{thm:ABD_Invariant_Spaces}
	Let $\mc{X}$ be a proper geodesic space and $\mf{S}$  an HHS structure for $\mc{X}$ with the bounded domain dichotomy. Suppose $\mf{T}$ is  the HHS structure produced by maximizing $\mf{S}$. The map $\phi\colon \partial ( \mc{X},\mf{S}) \to \partial (\mc{X},\mf{T})$ defined in Section \ref{subsec:defn_phi} is both a simplicial isomorphism $\partial_\Delta( \mc{X},\mf{S}) \to \partial_\Delta(\mc{X},\mf{T})$ and a homeomorphism $\partial ( \mc{X},\mf{S}) \to \partial (\mc{X},\mf{T})$. Moreover, the identity map $\mc{X} \to \mc{X}$ extends continuously to $\phi$.
\end{thm}

\begin{proof}

By Lemma \ref{lem:removebounded}, we can assume  $\mf{S} =\mf{S}_{ess}$ without loss of generality. We also assume $\pi_W(x)$ is a single point for each $x \in \mc{X}$ and $W \in \mf{S}$ by Standing Assumption \ref{assumption:projections_are_points}.

The fact that $\phi$ is a simplicial automorphism follows from the 
fact that we are assuming all the domains of $\mf{S}$ are essential 
and thus $\mf{S}$ and $\mf{T}$ have  identical sets of pairwise orthogonal
domains.

Define the map $\Phi\colon \mc{X}\cup\partial (\mc{X},\s)\to \mc{X}\cup\partial (\mc{X},\frak 
T)$ to be the identity on $\mc{X}$ and $\phi$ on the boundary.  As with $\phi$, we will denote the image $\Phi(p)$ by $\ol p$. 

The map $\Phi$ is a bijection, and we will show it is sequentially
continuous.  Since $\partial(\mc{X},\s)$ and $\partial(\mc{X},\frak
T)$ are first countable (this is implicit in \cite{HHS_Boundary} and 
explicitly proven in \cite[Proposition
1.5]{Hagen_second_countable}), sequential continuity implies that $\phi$
is continuous and is a continuous extension of the identity on 
$\mc{X}$.  Because
$\mc{X}$ is proper, $\mc{X}\cup\partial (\mc{X},\s)$ and
$\mc{X}\cup\partial (\mc{X},\frak T)$ are compact and Hausdorff
\cite[Theorem~3.4]{HHS_Boundary}.  Hence, proving $\phi$ is a continuous 
bijection implies it is a homeomorphism.

Let $p \in \partial (\mc{X},\s)$ and suppose $(p_n)$ is a sequence in $\mc{X} \cup \partial(\mc{X},\mf{S})$ that converges to $p$. To prove $\Phi$ is sequentially continuous, it suffices to prove
$\Phi(p_n)=\ol p_n\to \ol p=\Phi(p)$.

Continuing the convention from Notation \ref{not:twostructures}, we
let $\ol{\mc{B}}_{r,\varepsilon}(\cdot)$ denote a basis neighborhood
in $\mf{T}$ and $ \mc{B}_{r,\varepsilon}(\cdot)$ denote a basis
neighborhood in $\mf{S}$ as described in Definition \ref{defn:basisset}.  Recall, Definition
\ref{def:TnbhdsfromSnbhds} takes any neighborhood $M(r;p_W)$ of
$p_W\in \partial CW$ with $r \geq r_0$ and produces a neighborhood
$M(R_r;\ol p_{\ol W})$ of $\ol p_{\ol W} \in
\partial \mc{C} \ol W$ so that Lemmas \ref{lem:bdryprojopensets} and
\ref{lem:projopensetsnon} hold.  We let the constants $E$ and $C$
be as described in Notation \ref{notation:big_HHS_constant}.

Fix a basis neighborhood $\ol{\mathcal B}_{\ol r,\varepsilon}(\ol p)$
for some $\ol r\geq 0$ and $\varepsilon\geq 0$.  To prove $\ol p_n \to \ol p$, it suffices to show
that $\ol p_n\in \ol{\mathcal B}_{\ol r,\varepsilon}(\ol p)$ for all but
finitely many $n$.

By Lemma \ref{lem:nbhdsinnbhds}, there exists an $r$ sufficiently large   so
that the constant $R_r$ is defined and is large
enough  to ensure $\ol{\mc B}_{R_r,\varepsilon}(\ol p)\subseteq \ol{\mc
B}_{\ol r,\varepsilon}(\ol p)$.   Fixing this $r$,  it suffices to show that $\ol
p_n\in \ol{\mc B}_{R_r,\varepsilon}(\ol p)$ for all but finitely many
$n$.

Since $p_n\to p$, for each $s\geq 1$ and all but finitely many $n$, we have 
\[
p_n\in \mc B_{r+s,\frac1s}(p).
\] 

Notice that $M(r+s;p_W)\subseteq M(r;p_W)$ for all $s \geq 1$, and so $\mc B_{r+s,\frac1s}(p)\subseteq \mc B_{r,\varepsilon}(p)$ for all sufficiently large $s$.  In fact, it is clear from the definition of the decomposition of the neighborhoods that $\mc B^*_{r+s,\frac1s}(p)\subseteq \mc B^*_{r,\varepsilon}(p)$, where $*\in \{rem,int,non\}$.  Moreover, since $r+s\to \infty$ as $s\to\infty$, we have that $d_W(x_0,M(r+s;p_W))\to\infty$ as $s\to\infty$ for each $W\in \supp(p)$.

We divide the sequence $(\ol p_n)$ into three disjoint subsequences, and analyze each in a separate step of the proof.  We will show that for each such subsequence, $\ol p_n\in \ol{\mc B}_{R_r,\varepsilon}(\ol p)$ for all but finitely many $n$. \\

\noindent{\bf Step 1.} Consider the subsequence  consisting of all $n$ so that $p_n\in \mc B^{rem}_{r+s,\frac{1}{s}}(p)\subseteq \mc B^{rem}_{r,\varepsilon}(p)$.  If this subsequence is finite, we are done and move on to step two.  So suppose it is infinite.  There are two further subcases, depending on the support of $p$ and $\ol p$. Note that since $p_n$ is remote to $p$ in this case, we can apply Lemmas \ref{lem:bdryprojopensets} and \ref{lem:projopensetsnon} with $q = p_n$.  \\
 
	{\bf Step 1(a):} Assume $\supp(p)=\supp_{\frak T}(\ol p)$.  We  divide the elements of our subsequence of $p_n$ into two sets: those that satisfy $\supp(p_n)=\supp_\frak T(\overline p_n)$ and those which do not.    For the elements satisfying the first condition,  we will show that $\overline p_n\in \ol{\mc
	B}^{rem}_{R_r,\varepsilon}(\overline p)$.  For the remaining elements, we consider two further subcases, depending on whether $\supp(p)=\{S\}$ or not.  If it does, then we show that $\overline p_n\in \ol{\mc
	B}^{non}_{R_r,\varepsilon}(\overline p)$, and if it does not, then we show that $\overline p_n\in \ol{\mc
	B}^{rem}_{R_r,\varepsilon}(\overline p)$.  In any of the cases in this substep, we will have shown that $\overline p_n\in \ol{\mc
	B}_{R_r,\varepsilon}(\overline p)$, as desired.
	
	For
	any $n$ such that $\supp(p_n)=\supp_\frak T(\overline p_n)$, we
	must have $\ol p_n$ is remote to $\ol p$, because $p_n$ is remote
	to $p$ and remoteness is determined by the support set.  Hence, we
	can apply Lemma~\ref{lem:bdryprojopensets}\eqref{item:samesupp} to
	see that condition (R1) from Definition \ref{defn:remote_top} for
	$\ol{\mc{B}}^{rem}_{R_r,\varepsilon}(\ol p)$ is satisfied.  Since
	$\supp(p)^\perp = \supp_\mf{T}(\ol p)^\perp$, the fact that $p_n
	\in \mc{B}^{rem}_{r,\varepsilon}(p)$ implies (R2) and (R3) are
	also satisfied for $\ol{\mc{B}}^{rem}_{R_r,\varepsilon}(\ol p)$
	(to see this, note (R2) and (R3) involve $\varepsilon$ but not
	$r$).  This implies $\ol p_n\in \ol{\mc
	B}^{rem}_{R_r,\varepsilon}(\overline p)$, as desired.  
	
	We  now turn our attention to those $p_n$ in our subsequence which satisfy $\supp(p_n) \neq \supp_\frak T(\overline p_n)$.  In this case, we have $\supp_\frak T(\overline
	p_n)=\{S\}$ and $\supp(p_n)=\{Q_n\}$ for some $Q_n \in \mf{S} -
	\mf{T}$; in particular $Q_n \neq S$.  There are two cases to consider, depending on whether $\supp(p)=\{S\}$ or not.

First assume  $\supp(p)=\{S\}$.   Since, in this case,  $\supp_\frak T(\overline p)=\{S\}$,  it follows that $\overline p_n$ is not remote to $\overline p$.  We will check that $\overline p_n\in \ol{\mc B}^{non}_{R_r,\varepsilon}(\overline p)$.  Condition (N1) of Definition \ref{defn:nonremote_top} holds by  Lemma~\ref{lem:projopensetsnon}.   Condition (N2) holds as $a^{\ol p}_S=a^p_S=1$ and $a^{\ol p_n}_S=a^{p_n}_Q=1$.  Finally, (N3)  vacuously holds, as $$\supp_\frak T(\ol p_n)-(\supp_\frak T(\ol p_n)\cap \supp_\frak T(\ol p))=\emptyset.$$

Now assume $\supp(p)\neq\{S\}$. Since $\supp(p)=\supp_{\frak T}(\ol p)$, we also have $\supp_\frak T(\overline p)\neq \{S\}$.  This means $\overline p_n$ is remote to $\overline p$ because $\supp_\mf{T}(\ol p_n) = \{S\}$.  We will show that
$\overline p_n\in\mc B^{rem}_{R_r,\varepsilon}(\overline p)$ in this case.

  Since
$p_n\in \mc B^{rem}_{r,\varepsilon}(p)$, we have $\partial
\pi_P(p_n) \subseteq M(r;p_P)$ for all $P\in \supp(p)$.  By Lemma \ref{lem:bdryprojopensets}\eqref{item:samesupp}, we  therefore have
$\partial\overline\pi_P(\overline p_n) \subseteq M(R_r; \ol p_P)$
for all $P\in\supp(p)$, satisfying  (R1).  The condition (R3) is satisfied because $a^{\overline p_n}_W=0$ for all $W\in\supp_\frak T(\overline p_n)^\perp$, since $\supp_\frak T(\overline p_n)=\{S\}$ and $S$ is not orthogonal to any domain.

For the remaining condition (R2), note that  $\supp(p)^\perp = \supp_\frak T(\overline p)^\perp$. Because $\supp_\mf{T}(p)\neq\{S\}$, we have $\supp(p) = \supp_\mf{T}(\ol p) \subseteq \mf{T}$.   Thus, $\mc{C}_\mf{T}W = \mc{C}W$ for each $W \in \supp_\mf{T}(\ol p)$,  and  we can
think of both $\partial \pi_W(p_n)$ and $\partial\ol\pi_W(\ol p_n)$ as
subsets of $\mc CW$. Recall $ \supp_\mf{T}(\ol p_n) = \supp(p_n)$ contains only the domain $Q_n$. As in Definition \ref{def:remote}, let $\mc{S}_{\ol p_n}$ be the union of $\supp_\mf{T}(\ol p) = \supp(p)$ with the set of domains in $\supp(p)^\perp$ that are not orthogonal to $Q_n$ (this is the set of domains for which (R2) needs to be verified). By Lemma \ref{lem:boundaryprojintersect},   we have
\begin{equation}\label{eqn:partialbarpibd}
d_W(x_0,\partial \pi_W(p_n))-C \leq d_W(x_0,\partial\ol \pi_W(\ol p_n))\leq d_W(x_0,\partial \pi_W(p_n))+C
\end{equation}
for each $W \in  \mc{S}_{\ol p_n}$.
The following claim uses \eqref{eqn:partialbarpibd}  to complete the proof that (R2) holds for all but finitely many $n$.
\begin{claim}\label{claim:slopebound}
For any $W\in  \mc{S}_{\ol p_n}$ and $P\in \supp(p)$, we have 
\begin{equation}\label{eqn:slopebound}
\left|\frac{d_W(x_0,\partial\ol\pi_W(\ol p_n))}{d_P(x_0,\partial\ol\pi_P(\ol p_n))}-\frac{a^{\ol p}_W}{a^{\ol p}_P}\right|<\varepsilon
\end{equation}
for all but finitely many $n$.
\end{claim}

\begin{proof}
Recall that for any $s \geq 1$, we have that $p_n \in \mc
B_{r+s,\frac{1}{s}}(p)$ for all but finitely many $n$.  When $p_n \in \mc{B}^{rem}_{r+s,\frac{1}{s}}(p)$, (R1) and (R2) in $\mf{S}$
 imply for any 
$P\in\supp(p)$ and $W \in \mc{S}_{p_n}$ we have 
$$\partial\pi_P(p_n)) \subseteq M(r+s;p_P)$$
and 
\begin{equation}\label{eqn:R2inS}
\left|\frac{d_W(x_0,\partial\pi_W( p_n))}{d_P(x_0,\partial\pi_P( p_n))}-\frac{a^{ p}_W}{a^{p}_P}\right|<\frac{1}{s}.
\end{equation}
  Recall that $d_P(x_0,M(r+s;p_P))\to\infty$ as $s\to\infty$, which coupled with \eqref{eqn:partialbarpibd} implies that for all $P\in \supp(p)$, the distance $d_P(x_0,\partial\ol\pi_P(\ol p_n))\to\infty$  as $n\to \infty$.  

Suppose $W\in \supp(p)^\perp$ and  $d_W(x_0,\partial\ol\pi_W(\ol p_n))$ is uniformly bounded for all $n$.  In this case,  $a^P_W=0$, and since the numerator of the first term of \eqref{eqn:slopebound} is bounded while the denominator goes to infinity as $n\to\infty$, \eqref{eqn:slopebound} is satisfied for all but finitely many $n$.

Thus, we may assume without loss of generality that the numerator and denominator of the first term of  \eqref{eqn:slopebound} both go to infinity as $n\to \infty$.   Lemma~\ref{lem:changing_quotients_a_bounded_amount} now implies  \[\left|\frac{d_W(x_0,\partial\ol\pi_W(\ol p_n))}{d_P(x_0,\partial\ol\pi_P(\ol p_n))}-\frac{a^{\ol p}_W}{a^{\ol p}_P}\right|<\varepsilon\] because of \eqref{eqn:partialbarpibd}.
\end{proof}

 The above shows that $\ol p_n \in \ol{\mc{B}}^{rem}_{R_r,\varepsilon}(\ol p)$ for all but finitely many $n$  when $\supp(p) \neq \{S\}$. Combining this with the earlier proof that $\ol p_n \in \ol{\mc{B}}^{non}_{R_r,\varepsilon}(\ol p)$ when $\supp(p) = \{S\}$,  we conclude that $\overline p_n\in\ol{\mc B}_{R_r,\varepsilon}(\overline p)$ for all but finitely many $n$, whenever $\supp(p)=\supp_\frak T(\ol p)$.\\
	
	{\bf Step 1(b):} Assume $\supp(p)\neq\supp_{\frak T}(\ol p)$. Then $\supp_\mf{T}(\ol p) =\{S\}$ and $\supp(p) = \{P\}$ for some $P \in \mf{S} -\mf{T}$. We consider two subcases, depending on $\supp_\frak T(\overline p_n)$.
	\begin{itemize}
	\item Suppose $\supp_\frak T(\overline p_n)\neq\{S\}$ for some $n$.  Now $\overline p_n$ is remote to $\overline p$, because $\supp_\mf{T}(\ol p) = \{S\}$. Thus, we show $\ol p_n \in \ol{\mc{B}}_{R_r,\varepsilon}^{rem}(\ol p)$. Lemma \ref{lem:bdryprojopensets}\eqref{item:diffsupp}  shows that (R1) holds, and conditions (R2) and (R3)  vacuously hold as $\supp_\frak T(\overline p)^\perp=\emptyset$ and $|\supp_\frak T(\ol p)|=1$.  Therefore $\overline p_n \in \ol{\mc B}^{rem}_{R_r,\varepsilon}(\overline p)$.

	\item Suppose $\supp_\frak T(\overline p_n)= \{S\}$ and let $\supp(p_n)=\{Q\}$, where we include the possibility that $Q=S$.  In this  case, $\overline p_n$ is not remote to $\overline p$, so we show $\ol p_n \in \ol{\mc{B}}^{non}_{R_r,\varepsilon}(\ol p)$.  Condition (N1)  holds by  Lemma~\ref{lem:projopensetsnon}.   Condition (N2) trivially holds, as $\overline p_n=(\overline p_n)_S$  and $\overline p=\overline p_S$.  Finally, (N3)  is vacuously satisfied, as $\supp_\mf{T}(\ol p_n) = \supp_\mf{T}(\ol p) = \{S\}$.   Thus 	$\overline p_n\in \ol{\mc B}^{non}_{R_r,\varepsilon}(\overline p)$.
	\end{itemize}

\noindent{\bf Step 2.} Consider the subsequence consisting of all $n$ so that $p_n\in \mc B^{non}_{r+s,\frac{1}{s}}(p)\subseteq \mc B^{non}_{r,\varepsilon}(p)$.  If this subsequence is finite, we are done and move on to step three.  So suppose it is infinite.  There are two further subcases, depending on the support of $p$ and $\ol p$. \\

	{\bf Step 2(a):} Assume $\supp(p)=\supp_{\frak T}(\ol p)$.  First suppose that $\supp(p_n)\neq \supp_\frak T(\overline p_n)$ for some $n$. The only way for this to occur is for $\supp_\mf{T}(\ol p_n) = \{S\}$ and for $\supp(p_n) = \{Q\}$ for some $Q \in \mf{S} - \mf{T}$.  If $\supp(p)\cap\supp(p_n)=\emptyset$,  then $p_n$ is remote to $p$ because no domains are orthogonal to $Q$. However, this contradicts the assumption in this step,  so we must have $\supp(p)\cap\supp(p_n)\neq \emptyset$, and this intersection must be $\{Q\}$.  Hence $\supp(p)=\{Q\}$, because support sets are collections of pairwise orthogonal domains and there are no domains orthogonal to $Q$.  However, $Q\not\in\frak T$, and so $\supp_\frak T(\overline p)=\{S\}\neq\supp(p)$, which contradicts the assumptions of this case.  

Therefore we must have $\supp(p_n)=\supp_\frak T(\overline p_n)$, which makes $\ol p_n$ not remote to $\ol p$.   We verify that $\overline p_n\in \ol{\mc B}^{non}_{R_r,\varepsilon}(\overline p)$. If $\supp(p) = \supp_{\mf{T}}(\ol p) \subseteq \mf{T} -\{S\}$, then $\supp(p_n) = \supp(\ol p_n) $ is also contained in $\mf{T} -\{S\}$. This means $p_n = \ol p_n$, $p = \ol p$, and $\supp(p) \cap \supp(p_n) = \supp_{\mf{T}}(\ol p) \cap \supp_{\mf{T}}(\ol p_n)$. Hence, the conditions for $\ol p_n$ to be in  $\ol{\mc B}^{non}_{R_r,\varepsilon}(\overline p)$ follow from the fact that $p_n \in \mc{B}^{non}_{r,\varepsilon}(p)$  and $M(R_r;\ol p_W) \subseteq M(r,p_W)$ for each $W \in \supp(p)$ in this case. If instead $\supp(p) = \supp_{\mf{T}}(\ol p) =\{S\}$, then $\supp(p_n) = \supp(\ol p_n)=\{S\}$ as well. Thus (N1) is satisfied as $M(R_r;\ol p_S) \subseteq M(r,p_S)$ and (N2) and (N3) are trivially true because $a_S^{\ol p} = a_S^{\ol p_n} = 1$ and $\supp_\mf{T}(\ol p_n) - (\supp_\mf{T}(\ol p_n) \cap \supp_\mf{T}(\ol p)) = \emptyset$.\\
	
	{\bf Step 2(b):} Assume $\supp(p)\neq\supp_{\frak T}(\ol p)$.  In this case, $\supp_{\frak T}(\ol p)=\{S\}$ and  $\supp(p)=\{P\}$ for some $P \in \mf{S} -\mf{T}$.    Since $p_n$ is not remote to $p$, either $P\in\supp(p_n)$ or  $Q\perp P$ for all $Q\in \supp(p_n)$.  However, the latter is impossible,  as there are no domains orthogonal to $P$.  So we must have that $P\in\supp(p_n)$, in which case $\supp(p_n)=\{P\}$, as support sets are collections of pairwise orthogonal domains.  Thus $\supp_\frak T(\overline p_n)=\{S\}$, making $\ol p_n$ not remote to $\ol p$. 
	
	 We  show that $\overline p_n\in\ol{\mc B}^{non}_{R_r,\varepsilon}(\overline p)$. Since $p_n \in M(r,p_P) \cap \partial \mc{C}P$ and $P \in \mf{S} - \mf{T}$, Lemma \ref{lem:basisnbhd}\eqref{HypHHS:moving_sideways} applied to the hyperbolic HHS $(\mc{C}_\mf{T} S, (\mf{S} - \mf{T}) \cup \{S\} )$ says $\ol p_n \in M(R_r,\ol p_S)$. Hence Condition (N1) for $\ol p_n$ to be in  $\ol{\mc B}^{non}_{R_r,\varepsilon}(\overline p)$ is satisfied.  Condition (N3) is clear as $\overline p_n=(\overline p_n)_S$ and $\overline p=\overline p_S$, and Condition (N2) is vacuously satisfied as $\supp_\mf{T}(\ol p_n) = \supp_\mf{T}(\ol p)$. Therefore $\overline p_n\in\ol{\mc B}^{non}_{R_r,\varepsilon}(\overline p)$.\\

\noindent{\bf Step 3.} Consider the subsequence consisting of all $n$ so that $p_n\in \mc B^{int}_{r+s,\frac{1}{s}}(p)\subseteq \mc B^{int}_{r,\varepsilon}(p)$.  If this subsequence is finite, the proof is complete.  So suppose it is infinite.  There are two further subcases, depending on the support of $p$ and $\ol p$.   Since $\Phi$ restricts to the identity on $\mc{X}$, we have $\ol p_n=p_n$.\\

	{\bf Step 3(a):} Assume $\supp(p)=\supp_{\frak T}(\ol p)$.  Since $\supp(p)=\supp_\frak T(\overline p)$, we have $\supp(p)^\perp=\supp_\frak T(\overline p)^\perp$, and so (I2) and (I3) hold automatically because $p_n \in \mc{B}^{int}_{r,\varepsilon}(p)$.  If $\supp(p) = \supp_\mf{T}(\ol p) \neq \{S\}$, then (I1) follows from the fact that $\mc{C}W = \mc{C}_\mf{T}W$ for $W \in \supp(p)$, $R_r \leq r$, and $p_n \in \mc{B}^{int}_{r,\varepsilon}(p)$. If $\supp(p) = \supp_\mf{T}(\ol p) = \{S\}$, then $\mc{C}S \neq \mc{C}_\mf{T} S$. In this case, (I1) is a consequence of our choice of $R_r$ and Lemma \ref{lem:basisnbhd}\eqref{HypHHS:moving_up} applied to the hyperbolic HHS $(\mc{C}_\mf{T} S, (\mf{S} - \mf{T}) \cup \{S\} )$. Thus $p_n\in\ol{\mc B}^{int}_{R_r,\varepsilon}(\overline p)$.\\
	
	{\bf Step 3(b):} Assume $\supp(p)\neq\supp_{\frak T}(\ol p)$. 	In this case $\supp_{\mf{T}}(\ol p) = \{S\}$, which makes Conditions (I2) and (I3) for $p_n$ to be in $\ol{\mc{B}}^{int}_{R_r,\varepsilon}(\ol p)$ automatically true.  For Condition (I1),  we have $\pi_P(p_n)\subseteq M(r;p_P)$ for each $P \in \supp(p)$ because $p_n \in \mc{B}^{int}_{r,\varepsilon}(p)$. By our choice of $R_r$ (Definition~\ref{def:TnbhdsfromSnbhds}), this implies $\ol \pi_S(p_n) \subseteq M(R_r;\ol p _S)$, verifying (I1). \\

 We have show that for all but finitely many $n$, we have $\ol p_n\in \ol{\mc B}_{R_r,\varepsilon}(\ol p)\subseteq \ol{\mc B}_{r_0,\varepsilon}(\ol p)$.  Therefore $\ol p_n\to \ol p$ as $n\to\infty$ and $\Phi$ is sequentially continuous.
 This completes the proof of Theorem \ref{thm:ABD_Invariant_Spaces}.
 \end{proof}

\section{Applications}\label{sec:applications}

We conclude by recording several corollaries of Theorem
\ref{thm:BdryABDInvariant}, which show that some topological and
dynamical properties that are known to hold for maximized $G$--HHS
structures also hold in \emph{every} possible $G$--HHS structure for
the group.  We begin by recalling some definitions from
\cite{HHS_Boundary} about the dynamics of an element of a $G$-HHS on the HHS structure.  The
results we cite below were originally formulated in the setting of
hierarchically hyperbolic groups, but they continue to hold in the
slightly more general context of $G$--HHS because the definition of the HHS boundary does not interact in
any way with the domains whose associated hyperbolic spaces have
uniformly bounded diameter.  Also, since any HHG is a $G$--HHS, the
results in this section imply all the statements from the introduction
about hierarchically hyperbolic groups.

Fix a $G$--HHS $(G,\mf{S})$ with $\nest$--maximal element $S \in \mf{S}$.  The
\emph{big set in $\mf{S}$} of an element $g\in G$ is the set of
domains $W\in\mf{S}$ so that $\diam(\pi_W(\langle g \rangle)) =
\infty$; we denote the big set in $\mf{S}$ by $\BS_\mf{S}(g)$.  We say
$g$ is \emph{irreducible} with respect to $\mf{S}$ if $\BS_\mf{S}(g) =
\{S\}$.  If $\BS_{\mf{S}}(g) \neq \emptyset$ but $\BS_{\mf{S}}(g) \neq
\{S\}$, then we say $g$ is \emph{reducible} with respect to $\mf{S}$.
Durham, Hagen, and Sisto show the following basic properties of the
big sets, which holds more generally assuming that $(G,\mf{S})$ is 
only a $G$--HHS.

\begin{lem}[\cite{HHS_Boundary}]\label{lem:big_sets}
	Let $(G,\mf{S})$ be an HHG and $g\in G$.
	\begin{enumerate}
		\item $\BS_{\mf{S}}(g) = \emptyset$ if and only if $g$ has finite order. That is, every element of $G$ is either irreducible, reducible, or finite order.
		\item $\BS_{\mf{S}}(g)$ is a pairwise orthogonal subset of domains of $\mf{S}$. In particular, there is $k \in\mathbb{N}$ depending only on $\mf{S}$  so that $g^k W = W$  for each  $W \in \BS_{\mf{S}}(g)$.
		\item For each $n \in\mathbb{Z}$ and $W \in \BS_{\mf{S}}(g)$, if $g^n W =W$, then $g^n$ acts loxodromically on $\mc{C} W$.
	\end{enumerate}
\end{lem}

  Recall that the action of $G$ on itself is by hieromorphisms where each of the maps between hyperbolic spaces in an isometry; see Section \ref{subsec:boundary_maps}. Hence, the action of $G$ on itself extends continuously to an action by both homeomorphisms and simplicial automorphism on the boundary. The main dynamical property  of this action that we will be interested is that of north-south dynamics.

\begin{defn} 
	Let $(G,\mf{S})$ be a $G$--HHS that is not virtually cyclic.
	An element $g \in G$ acts with \emph{north-south dynamics} on $\partial (G,\mf{S})$ if $g$ has exactly two fixed points $\xi^+,\xi^- \in \partial (G,\mf{S})$ and for any disjoint open sets $O^+,O^- \subseteq \partial (G,\mf{S})$  containing $\xi^+$ and  $\xi^-$  respectively, there exists $n \in \mathbb{N}$ so that $$g^n \cdot (\partial(G,\mf{S}) - O^- )\subseteq O^+.$$  We call $\xi^+$ the \emph{attracting fixed point} of $g$ and $\xi^-$ the \emph{repelling fixed point}.
\end{defn}

\begin{rem}\label{rem:more_than_3_boundary_points}
	 Combining the work in \cite{HHS_Boundary} and \cite{ABD} yields: $\partial(G,\mf{S}) = \emptyset$ if and only if $G$ is
	 finite; $|\partial (G,\mf{S})| = 2$ if and only if $G$ is
	 virtually $\mathbb Z$; and $|\partial (G,\mf{S})| = \infty$ in
	 all other cases.  Thus $|\partial (G,\mf{S})| =\infty$ if and
	 only if $G$ is not virtually cyclic.
\end{rem}

Durham, Hagen, and Sisto show that the irreducible elements always act
with north-south dynamics, and that the set of attracting fixed points
of the irreducible elements is dense in the HHS boundary; again this 
holds equal as well under the assumption that $G$ is a $G$--HHS.

\begin{thm}[\cite{HHS_Boundary}]\label{thm:DHS_density}
	Let $(G,\mf{S})$ be a non-virtually cyclic HHG  and $S \in \mf{S}$ be the $\nest$--maximal element. If $\mc{C}S$ has infinite diameter, then
	\begin{enumerate}
		\item if $g \in G$ is irreducible with respect to $\mf{S}$, then $g$ acts with north-south dynamics on $\partial (G,\mf{S})$;
		\item the set of attracting fixed points for the irreducible elements of $(G,\mf{S})$ are dense in $\partial(G,\mf{S})$;
		\item the inclusion of $\partial \mc{C}S$ into $\partial (G,\mf{S})$ is a continuous embedding whose image is dense in $\partial (G,\mf{S})$.
	\end{enumerate}
\end{thm}

A result of the first two authors and Durham established that the
hyperbolic space obtained from the maximization procedure is
independent of the initial HHG structure (this result is implicit in
\cite[Theorem~5.1]{ABD}, which proves that the hyperbolic space
associated to the nest-maximal element in a maximized structure is the
initial object in a particular category and whence unique). Since the proof in \cite{ABD} only
involves the domains in $\mf{S}^\infty$, the proof there for HHGs also 
establishes the identical result for 
$G$--HHSs. 

\begin{thm}[\cite{ABD}]\label{thm:unqiue_maximized_hyp_space}
	Let $\mf{S}_1$ and $\mf{S}_2$ be two HHG structures for the
	group $G$.  Let $\mf{T}_1$ and $\mf{T}_2$ be the maximizations of
	$\mf{S}_1$ and $\mf{S}_2$ accordingly.  If $S_1$ and $S_2$ are the
	$\nest$--maximal elements of $\mf{T}_1$ and $\mf{T}_2$
	respectively, then $\mc{C}S_1$ is quasi-isometric to $\mc{C}S_2$.
	In this case, $\partial \mc{C}S_1$ is homeomorphic to $\partial
	\mc{C} S_2$.
\end{thm}

Combining this uniqueness with the density results from Theorem
\ref{thm:DHS_density}, some topological properties of the boundary of
the maximized hyperbolic space can be expanded to the HHS boundary of 
any $G$--HHS.  One salient example of such a topological property that can
be extended from a dense subset is connectedness.

\begin{cor}\label{cor:boundary_connected}
	Let $(G,\mf{T})$ be a maximized $G$--HHS.  If $T \in\mf{T}$ is the $\nest$--maximal element and
	$\mc{C}T$ is one-ended, then for any $G$--HHS structure $\mf{S}$
	for $G$, the HHS boundary $\partial(G,\mf{S})$ is
	connected.
\end{cor}

\begin{proof}
	Let $\mf{S}$ be a $G$--HHS structure for $G$ and let $\mf{R}$ be the
	maximization of $\mf{S}$.  Let $S$ be the $\nest$--maximal element
	of $\mf{S}$ and $\mf{R}$. 
	
		Since $G$ acts coboundedly on $\mc{C}T$, one-endedness of
	$\mc{C}T$ is equivalent to connectedness of $\partial \mc{C} T$.
	Now $\mc{C}T$ is quasi-isometric to
	$\mc{C}_\mf{R} S$ by Theorem \ref{thm:unqiue_maximized_hyp_space},
	so $\partial \mc{C}T$ being connected implies $\partial
	\mc{C}_\mf{R}S$ is connected.  Since $\partial \mc{C}_\mf{R}S$ is
	dense in $\partial (G,\mf{R})$, this implies $\partial (G,\mf{R})$
	is also connected.  Thus, $\partial (G,\mf{S})$ is connected by
	Theorem \ref{thm:BdryABDInvariant}.
\end{proof}

\begin{rem}
	Unfortunately, such an argument with density can not be used to
	show the HHS boundary is path connected. The topologist's sine
	curve is an example of a space that has a dense path-connected
	subset, but is not path-connected.
\end{rem}

Our next set of applications involve the Morse elements of a $G$--HHS and the elements that act by north-south dynamic on the boundary.

A quasigeodesic $\gamma$ in a metric space is \emph{Morse} if there
exists a function $N \colon [1,\infty) \times [0,\infty) \to
[0,\infty)$ so that each $(\lambda,c)$--quasigeodesic with endpoints
on $\gamma$ is contained in the $N(\lambda,c)$--neighborhood of
$\gamma$.  An element $g$ of a finitely generated group $G$ is
\emph{Morse} if $\langle g \rangle$ is a Morse quasigeodesic in the
Cayley graph of $G$ with respect to some finite generating set.  Since
a quasigeodesic being Morse is preserved by quasi-isometries, whether
or not $g \in G$ is Morse is independent of the choice of finite
generating set for $G$.

One of the original applications of the maximization procedure was to
characterize Morse elements of an HHG as precisely those that
are irreducible with respect to a maximized structure. Since this 
condition involves only  domains in $\mf{S}^\infty$, its proof also works 
identically for $G$--HHSs.

\begin{thm}[\cite{ABD}]\label{thm:Morse_iff_irr_in_max}
	Let $G$ be a HHG. If $\mf{T}$ is a maximized structure,
	then $g\in G$ is Morse if and only if $g$ is irreducible with
	respect to $\mf{T}$.
\end{thm}

Since the boundary is invariant under maximization and the maximized
hyperbolic space is unique up to quasi-isometry, we now show that the
Morse elements are precisely the set of elements that act with
north-south dynamics on the HHS boundary of \textit{any} $G$--HHS.

\begin{cor}\label{cor:boundary_NSdynamics}
	Let $(G,\mf{S})$ be a $G$--HHS that is not virtually cyclic. An element $g \in G$ acts with north-south dynamics on $\partial(G,\mf{S})$ if and only if $g$ is a Morse element of $G$. In particular, the set of elements of $G$ that act with north-south dynamics does not depend on the $G$--HHS structure $\mf{S}$.
\end{cor}

\begin{proof}
	Let $\mf{T}$ be the maximization of the $G$--HHS structure $\mf{S}$. Let $S$ be the $\nest$--maximal domain of $\mf{S}$ and $\mf{T}$. 
	 Since $G$ is not virtually cyclic,  both $\partial(G,\mf{T})$ and  $\partial(G,\mf{S})$ have an infinite number of points; see Remark \ref{rem:more_than_3_boundary_points}. 
	
	Assume that $g \in G$ is a Morse element. By Theorem \ref{thm:Morse_iff_irr_in_max}, $g$ being Morse is equivalent to $g$ being irreducible with respect to $\mf{T}$. Hence, $g$ acts with north-south dynamics on $\partial (G,\mf{T})$ (Theorem \ref{thm:DHS_density}). Since $\partial (G,\mf{T})$ is $G$--equivariantly homeomorphic to $\partial (G,\mf{S})$ by Theorem \ref{thm:BdryABDInvariant}, $g$ must also act on $\partial(G,\mf{S})$ with north-south dynamics. 
	
	Now assume $g\in G$ acts by north-south dynamics on $\partial (G,\mf{S})$. Since $\partial (G,\mf{S})$ is Hausdorff and has an infinite number of points, north-south dynamics ensures that $g$ does not have finite order. Hence,   $\BS_{\mf{S}}(g) \neq \emptyset$ by Lemma \ref{lem:big_sets}. Further, \cite[Proposition 6.22]{HHS_Boundary} says $|\BS_{\mf{S}}(g)| > 1$  implies $g$ stabilizes at least $3$ points in $\partial(G,\mf{S})$, which would contradict $g$ having north-south dynamics. Thus, we know $\BS_{\mf{S}}(g)$ contains exactly one domain $W \in \mf{S}$. Since $gW = W$, if $V \perp W$, then $gV \perp W$ as well. Thus,   $\mf{S}_W^\perp \cap \mf{S}^\infty \neq \emptyset$ would imply that the non-empty set of points $\{p \in \partial (G,\mf{S}): \supp(p) \subseteq \mf{S}_W^\perp \cap \mf{S}^\infty\}$ is stabilized by $g$. Since this would violate north-south dynamics, we know    $\mf{S}_W^\perp \cap \mf{S}^\infty =\emptyset$. Hence, $W \not \in \mf{T}-\{S\}$ as maximization removes all non-maximal domains that are not orthogonal to a domain of $\mf{S}^\infty$. This implies $\BS_{\mf{T}}(g) =\{S\}$, which makes $g$ a Morse element by Theorem \ref{thm:Morse_iff_irr_in_max}.
\end{proof}

Since the set of attracting fixed points for the irreducible elements
are dense, we have the same for the attracting fixed points of the
Morse elements regardless of the choice of $G$--HHS structure.

\begin{cor}\label{cor:Morse_boundary_is_dense}
	Let $(G,\mf{S})$ be a $G$--HHS that is not virtually cyclic. If $G$ contains a Morse element, then the set of attracting fixed points of Morse elements in $\partial(G,\mf{S})$  is dense in $\partial (G,\mf{S})$.
\end{cor}

\begin{proof}
	Let $\mf{T}$ be the maximization of $\mf{S}$. By Theorem \ref{thm:DHS_density} and Theorem \ref{thm:Morse_iff_irr_in_max}, the set of attracting fixed point of the Morse elements is dense in $\partial(G,\mf{T})$. Theorem \ref{thm:BdryABDInvariant} then implies that they are also dense in $\partial(G,\mf{S})$.
\end{proof}

\begin{rem}[Density of the Morse boundary]
The Morse geodesics of a group can be used to make a \emph{Morse
boundary} for any finitely generated group; see \cite{Cordes_Survey}.
The Morse boundary of an HHG has a $G$--equivariant continuous
injection into $\partial (G,\mf{T})$ where $\mf{T}$ is a maximized
structure \cite[Theorem 6.6]{ABD} (this result again works for $G$--
HHSs, as well).  Every Morse element
also has a pair of fixed points in the Morse boundary, and the
continuous inclusion sends fixed points to fixed points.  Hence,
Corollary \ref{cor:Morse_boundary_is_dense} shows that the image of
the Morse boundary in the HHS boundary is dense.
\end{rem}

Finally, we use the density of Morse elements to show the limit set of
a normal subgroup is the entire HHS boundary.  The \emph{limit
set} of a subgroup $H$ of a $G$--HHS $(G,\mf{S})$ is the set of all
points in $\partial (G,\mf{S})$ that are the limit of a sequence of
elements of $H$.

\begin{cor}\label{cor:normal_subgroups}
	Let $(G,\mf{S})$ be a $G$--HHS that is not virtually cyclic.  If $G$ 
	contains a Morse element and $N$
	is an infinite normal subgroup of $G$, then
	the limit set of $N$ in $\partial (G,\mf{S})$ is all of $\partial
	(G,\mf{S})$.
\end{cor}

\begin{proof}
	Let $\mf{T}$ be the maximization of $\mf{S}$, and let $S$ be the
	$\nest$--maximal element of $\mf{S}$ and $\mf{T}$.  Since
	$\partial \mc{C}_\mf{T} S$ is dense in $\partial (G,\mf{T})$ and
	by Theorem \ref{thm:BdryABDInvariant} the identity $G \to G$
	induces a homeomorphism $G \cup \partial (G,\mf{S}) \to G \cup
	\partial(G,\mf{T})$, it suffices to prove that the limit set of
	$N$ contains all of $\partial \mc{C}_\mf{T} S \subseteq \partial
	(G,\mf{T})$.
	
	Let $p \in \partial(G,\mf{T})$ with $\supp(p) = \{S\}$.  Fix a
	basis neighborhood $\mc{B}_{r,\varepsilon}(p)$ in $\partial
	(G,\mf{T})$.  Since $G$ has at least one Morse element, Corollary
	\ref{cor:Morse_boundary_is_dense} says there is a Morse element
	$g\in G$, so that its attracting fixed point $\xi$ is
	contained in $\mc{B}_{r,\varepsilon}(p)$.   Proposition \ref{prop:Morse_in_normal} below shows that there is also a Morse element $h$ in
	the normal subgroup $N$ with attracting fixed point $\zeta \in
	\mc{C}_\mf{T} S \subseteq \partial(G,\mf{S})$.  Since the Morse
	elements of $G$ act with north-south dynamics, there is some $n
	\in\mathbb{N}$ so that $g^n \zeta \in \mc{B}_{r,\varepsilon}(p)$.
	Now, the sequence $\{ g^n h^i g^{-n}\}_{i=1}^\infty$ will converge
	to $g^n \zeta$ because $\{\pi_S(g^n h^i g^{-n})\}_{i=1}^\infty$
	will converge to $g^m \zeta \in \partial \mc{C}_\mf{T}S$.  Since
	$h \in N$ and $N$ is a normal subgroup, $g^n h^i g^{-n} \in N$ for
	each $i \in \mathbb{N}$.  Thus, $p$ will be a limit point of
	elements of $N$.
\end{proof}

\begin{rem}
	The conclusion of Corollary \ref{cor:normal_subgroups} can fail 
	to hold when $G$ does
	not contain any Morse elements.  For example, if $G$ is the direct
	product of two infinite $G$--HHSs $H_1 \times H_2$, then the HHS 
	boundary of $G$ will be the join of the HHS boundaries of $H_1$
	and $H_2$, and the limit set of each $H_i$ will be exactly one
	side of this join.
\end{rem}

Our last proposition was used in the proof of Corollary
\ref{cor:normal_subgroups} and is useful in its own right. We note 
that it can be deduced as a special case of
\cite[Corollary 3.6]{Morse-local-to-global}; we provide a short proof 
using the theory of hierarchical hyperbolicity for completeness.  

\begin{prop}\label{prop:Morse_in_normal} Let $G$ be a $G$--HHS  containing a Morse element.  Then every infinite
normal subgroup of $G$ contains an element that is Morse in~$G$.
\end{prop}

\begin{proof} Let $g\in G$ be a Morse element, and let $\mf{S}$ be a
maximized structure for $G$.  By Theorem
\ref{thm:Morse_iff_irr_in_max}, $\BS_{\mf{S}}(g)=\{S\}$ where $S$ is
the $\nest$--maximal element of $\mf{S}$.  Let $N$ be a normal
subgroup of $G$.  Applying the Rank-Rigidity Theorem (\cite[Theorem
9.13]{HHS_Boundary}) to the action of $N$ on $(G,\mf{S})$, either $N$
contains a Morse element or it stabilizes a product region
$\mathbf{P}_U$ for some $U \in \mf{S}$ ($\P_U$ is the image of $\F_U
\times \E_U$ in $G$).  Since $g$ is loxodromic on $\mc{C}S$, by taking
$n$ large enough we can ensure that $\rho^{U}_{S}$ is as far as
desired from $\rho^{g^{n}U}_{S}$.  Since $\rho^{U}_{S} \asymp
\pi_S(\mathbf{P}_U)$, the orbits $N \cdot \rho^N_S$ and $g^nNg^{-n}
\cdot \rho^{g^nU}_S$ both have uniformly bounded diameter for any $n
\in \mathbb{Z}$.  Since $g^n N g^{-n} = N$, acylindricity of the
action of $G$ on $\mc{C}S$ implies that $N$ is finite.
\end{proof}

\bibliography{HHS_Boundary}{}
\bibliographystyle{alpha}

\end{document}